\documentclass[12pt]{amsart}
\usepackage{amscd,amssymb,amsthm,verbatim,slashed}
\newcommand{\Alt}{\operatorname{Alt}}

\newcommand{\Aut}{\operatorname{Aut}}

\newcommand{\const}{\operatorname{const.}}

\newcommand{\del}{\operatorname{\slashed{\partial}}}

\newcommand{\Dom}{\operatorname{Dom}}
\newcommand{\dvol}{\operatorname{dvol}}

\newcommand{\GL}{\operatorname{GL}}

\newcommand{\HH}{\operatorname{H}}

\newcommand{\Image}{\operatorname{Im}}

\newcommand{\Int}{\operatorname{int}}

\newcommand{\Ker}{\operatorname{Ker}}

\newcommand{\OO}{\operatorname{O}}

\newcommand{\R}{{\mathbb R}}

\newcommand{\SO}{\operatorname{SO}}

\newcommand{\Spin}{\operatorname{Spin}}

\numberwithin{equation}{section}
\theoremstyle{plain}
\newtheorem{definition}{Definition}
\newtheorem{assumption}{Assumption}
\newtheorem{lemma}{Lemma}
\newtheorem{theorem}{Theorem}
\newtheorem{proposition}{Proposition}

\newtheorem{property}{Property}

\theoremstyle{remark}

\newtheorem{remark}{Remark}
\newtheorem{example}{Example}
\usepackage{graphicx}
\input{amssym.def}
\input{amssym}
\input{epsf}
\usepackage{a4wide}

\begin{document}

\title[Scalar curvature and Dirac operator for some singular spaces]
      {Scalar curvature and Dirac operator for some singular spaces}

\author{John Lott}
\address{Department of Mathematics\\
University of California, Berkeley\\
Berkeley, CA  94720-3840\\
USA} \email{lott@berkeley.edu}

\begin{abstract}
We look at smooth manifolds equipped with a possibly singular Riemannian metric.  We give
sufficient conditions for the existence of scalar curvature measures and Dirac operators.
\end{abstract}

\date{December 22, 2025}

\maketitle

\section{Introduction} \label{sect1}

In studying scalar curvature for singular spaces, there are two distinct but related questions.  One is to make sense of the notion of a lower bound on the scalar curvature.  The other is to make sense of the scalar curvature as a generalized function.

The paradigm comes from convex functions on
$\R$. One can define a smooth convex function by saying that the second derivative is nonnegative.  Passing to pointwise limits gives rise to the notion of a convex function in the derivative-free sense.  To close the circle, a convex function on $\R$ has a distributional second derivative that is a nonnegative measure.

Analogously, in the geometric setting one can ask in what generality one can make sense of a lower curvature bound or, alternatively, of a generalized curvature tensor.  This has been explored for spaces with lower bounds on the
sectional curvature (\cite{AKP (2024),Lebedeva-Petrunin (2024)} and references therein)
and spaces with lower bounds on the Ricci curvature
(\cite{Honda (2014),Lott (2016)} and references therein).

For scalar curvature, it has become clear that the right notion is a scalar curvature measure, rather than a scalar curvature function.  In the smooth case, if the scalar curvature function is denoted by $R$ then the scalar curvature measure $dR$ is $R \dvol$. In \cite{Lott (2022)} we looked at the scalar curvature measure
for certain limit spaces, using Ricci flow.
In this paper we consider possibly singular
Riemannian metrics on smooth manifolds and examine when a scalar curvature measure is well defined.

For smooth manifolds, the main obstructions to the existence of smooth metrics with positive scalar curvature 
come from
Dirac operators and minimal hypersurfaces. It's reasonable to ask to what extent these tools extend to
singular spaces.  In the smooth case, the link between the Dirac operator $D$ and the scalar curvature measure $dR$ comes from the integrated
Lichnerowicz identity 
\begin{equation} \label{1.1}
\int_M \left( \langle D \eta, D \psi \rangle - \langle \nabla \eta, 
\nabla \psi \rangle \right) \dvol_g = \frac14 \int_M \langle \eta, \psi \rangle dR.
\end{equation}
In the singular case, one can ask if one can define $dR$ this way, letting $\eta$ and $\psi$ range over smooth spinor fields. Going beyond formal calculations, one can also ask when $D$ is well defined as 
a densely defined self-adjoint operator on a Hilbert space of $L^2$-spinors.

To describe the results of the paper, we begin with a smooth manifold $M$ equipped with a Riemannian metric $g$, not necessarily smooth. 
One setting in which the scalar curvature of $g$ makes sense, as a 
distribution, is when $g$ is biLipschitz to a Euclidean metric in local coordinates, and has
first derivatives that are locally square integrable. While this is useful for some purposes, as in the work of
Lee and LeFloch \cite{Lee-LeFloch (2015)}, it was already pointed out by
Geroch and Traschen \cite{Geroch-Traschen (1987)} that the hypotheses are not satisfied by a two
dimensional cone metric.  Since the latter is precisely where one would expect to see a nontrivial
scalar curvature measure, namely a multiple of the delta function at the vertex, one would like to have
a framework that accounts for this.  

We first look at the generality in which the formula for the scalar curvature measure, in the smooth
case, extends to define a distributional scalar curvature.  It's convenient to work with a local orthonormal
frame $e_i = \sum_\mu e_i^{\: \: \mu} \partial_\mu$, where Greek letters denote coordinate indices and Latin letters denote orthonormal frame
indices. Let $\omega^i_{\: \: j} = \sum_\mu \omega^i_{\: \: j \mu} dx^\mu$ be the connection $1$-forms.
We use the Einstein summation convention.

\begin{theorem} \label{thm1}
The formula (\ref{3.1}) gives a well-defined scalar curvature distribution $dR$ on $M$ if in local coordinates,
the following
expressions are $L^1$ with respect to $\dvol_{g}$:
\begin{enumerate}
\item $e_a^{\: \: \mu} e_b^{\: \: \nu} - e_b^{\: \: \mu} e_a^{\: \: \nu}$.
\item $\left( e_a^{\: \: \mu} e_b^{\: \: \nu} - 
e_a^{\: \: \nu} e_b^{\: \: \mu}
\right) \omega^{ab}_{\: \: \: \: \nu} $ and
$- e_a^{\: \: \mu} e_c^{\: \: \nu} \omega^{\: \: c}_{b \: \: \: \nu} +
e_b^{\: \: \mu} e_c^{\: \: \nu} \omega^{\: \: c}_{a \: \: \: \nu}
+ e_c^{\: \: \mu} e_b^{\: \: \nu} \omega^c_{\: \: a\nu} - e_c^{\: \: \mu} e_a^{\: \: \nu} \omega^c_{\: \: b\nu}$.
\item 
$e_a^{\: \: \mu} e_b^{\: \: \nu}
(
\omega^a_{\: \: c \mu} \omega^{cb}_{\: \: \: \:  \nu} - \omega^{a}_{\: \: c \nu} \omega^{cb}_{\: \: \: \:  \mu}
)$.
\end{enumerate}
\end{theorem}

The conditions in Theorem \ref{thm1} are gauge invariant.  They are satisfied by a cone with a smooth 
spherical link in
any dimension.

A more geometric way to construct a scalar curvature distribution is to use equation (\ref{1.1}). 
In \cite{Lott (2016)} we did something similar for the Ricci tensor, using a Bochner formula.

We assume that $M$ is spin, with spinor bundle $SM$, although the next
result is local and so the spin assumption is not important.
Let $\Alt$ denote antisymmetrization of a tensor.

\begin{theorem} \label{thm2}
The left-hand side of (\ref{1.1}) is well defined for all $\eta, \psi \in C^\infty_c(SM)$ if and only if in local 
coordinates, the following expressions are $L^1$ with respect to $\dvol_{g}$:
\begin{enumerate}
\item $e_c^{\: \: \mu} e_d^{\: \: \nu} - e_d^{\: \: \mu} e_c^{\: \: \nu}$.
\item $\left( e_a^{\: \: \mu} e_b^{\: \: \nu} - 
e_a^{\: \: \nu} e_b^{\: \: \mu}
\right) \omega^{ab}_{\: \: \: \: \nu} $,
$- e_a^{\: \: \mu} \omega_{b\: \: d}^{\: \: d}
+ e_b^{\: \: \mu} \omega_{a\: \: d}^{\: \: d} + e_c^{\: \: \mu}\omega^c_{\: \: ab} - e_c^{\: \: \mu}\omega^c_{\: \: ba}$ and
$\Alt_{cdef} (e_c^{\: \: \mu} \omega_{efd})$.
\item $e_a^{\: \: \mu} e_b^{\: \: \nu}
(
\omega^a_{\: \: c \mu} \omega^{cb}_{\: \: \: \:  \nu} - \omega^{a}_{\: \: c \nu} \omega^{cb}_{\: \: \: \:  \mu}
)$ and
$\Alt_{mnpq} \left(\omega_{m \: \: r}^{\: \: r} \omega_{pqn} - 
\omega_{pqr} \omega^r_{\: \: nm} + \omega_{mrp} \omega^r_{\: \: nq}
\right)$.
\end{enumerate}
\end{theorem}

Note that the conditions in Theorem \ref{thm1} are a subset of those in Theorem \ref{thm2}.
It is possible that the left-hand side of (\ref{1.1}) makes sense even if
$\int_M \langle  D \eta, D \psi \rangle \dvol_g$ and
$\int_M \langle \nabla \eta, 
\nabla \psi \rangle \dvol_g$ do not make sense individually; this is what happens for a
two dimensional cone.
We compute the scalar curvature measure $dR$ when two Riemannian manifolds are glued together (Proposition \ref{prop5}), and for a family of conical metrics (Proposition \ref{prop6}).

The next main topic is Dirac operators on singular spaces.  The goal is to construct and analyze such operators as  densely defined self-adjoint operators on a Hilbert space. Dirac operators have been considered on spaces with
isolated conical singularities since the work of Chou \cite{Chou (1985)}. However, some of the relevant spaces arising
as geometric limits have dense sets of conical singularities. It is not so clear in such a case how to adapt the
known methods.  The simplest case where this can happen is a surface $M$ with Alexandrov curvature bounded below.
Such a surface has a metric that can be written as $g = e^{2 \sigma} g_0$ where
$g_0$ is a smooth metric on $M$.  Assume that the surface is spin and compact.  One may hope to use the smooth spinors on $M$ as a dense domain for a Dirac operator $D$.  However, it turns out that $D$ generally doesn't map smooth spinors to $L^2$-spinors, so
smooth spinors are not an appropriate domain. 

If $D_0$ denotes the Dirac operator associated to $g_0$ then
formally, $D = e^{- 3 \sigma/2} D_0 e^{\sigma/2}$. This motivates instead using $e^{- \sigma/2} C^\infty(SM)$ as a
domain for a Dirac operator.  Equivalently, we conjugate by $e^{\sigma/2}$ and consider the
operator $\widehat{D} = e^{\sigma/2} D e^{- \sigma/2} = e^{- \sigma} D_0$, with domain $C^\infty(SM)$.
The new $L^2$-inner product on $SM$ is $\langle \eta,\psi \rangle_{\mathcal H} = 
\int_M \langle \eta, \psi \rangle e^\sigma \dvol_{g_0}$.
If $\sigma$ is smooth then $\widehat{D}$ is isometrically equivalent to $D$.

We assume that $e^{2 \sigma} \in L^1(M, \dvol_{g_0})$ (i.e. finite area) and $e^{-\sigma} \in L^1(M, \dvol_{g_0})$.
Let $\widehat{\partial}_\pm$ be the restriction of $\widehat{D}$ to smooth $\pm$-spinors. It is closable.
Let $\widehat{\partial}_{\pm, max}$ be its maximal extension.
Let $K^\frac12$ be the square root of the canonical bundle, corresponding to the spin structure,
and let $\HH^*$ denote sheaf cohomology groups.

\begin{theorem} \label{thm3}
$\widehat{\del}_{\pm, max}$ is Fredholm.  Its index is zero.  
There are isomorphisms $\Ker \left( \widehat{\del}_{+, max}\right) \cong \Ker \left( \widehat{\del}_{-, max}^* \right) \cong \HH^0(M, K^\frac12)$ and
$\Ker \left( \widehat{\del}_{-, max}\right) \cong \Ker \left( \widehat{\del}_{+, max}^* \right)
\cong \HH^1(M, K^\frac12)$.
\end{theorem}

The issues in defining a self-adjoint Dirac operator are not special to two dimensions; the same issues will
arise in higher dimension for the product of an Alexandrov surface with a smooth compact manifold.  Instead, they are
codimension-two phenomena. As the use of isothermal coordinates and conformal factors is special to two
dimensions, we show that $\widehat{D}$ can be equivalently described as the Dirac operator acting on spinors
coupled to the quarter density bundle ${\mathcal D}^\frac14$ of the surface.  This latter description makes
sense in any dimension. 

We now consider a compact spin manifold of dimension $n$ with a possibly singular Riemannian metric.
In local coordinates, let $|g|$ denote the determinant $\det(g_{\mu \nu})$ of the metric tensor.
Suppose that in a local coordinate system and with respect to a local orthonormal frame, the quantities 
$e_a^{\: \: \mu}$, 
$\omega_{abc}$ and $e_a^{\: \: \mu} \partial_\mu \log |g|$ are $L^2$ with respect to $|g|^\frac14 dx^1 \ldots dx^n$.
Then $\widehat{D}$ is well defined as a symmetric operator acting on smooth sections of $SM \otimes {\mathcal D}^\frac14$. It is not so clear in what generality the operator is essentially self-adjoint.  If $n$
is even then there
is a least one self-adjoint extension for the following reason.  Let $\widehat{\partial}_\pm$ be the restriction of
$\widehat{D}$ to smooth $\pm$-spinors coupled to ${\mathcal D}^\frac14$. Then $\widehat{\partial}_\pm$ is closable
and we have the self-adjoint extension(s) $\widehat{\partial}_{\pm,min} + \widehat{\partial}_{\pm,min}^*$
of $\widehat{D}$. In this generality, there
is a Lichnerowicz vanishing result as follows, where the scalar curvature $\widehat{dR}$ is now a distributional
half density.

\begin{theorem} \label{thm4}
If $\widehat{dR}$ is positive then $\Ker(\partial_{\pm,min}) = 0$.
\end{theorem}

Finally, in a different direction, we consider the scalar curvature measure in harmonic coordinates.
It turns out that the second derivatives of the metric only enter through the Laplacian of $\log |g|$. This has the implication
that we can get information about the location of the singularities of the scalar curvature measure, from just the metric tensor, without taking derivatives.

\begin{theorem} \label{thm5}
Suppose that the distribution $dR$ is a measure.
Let $U$ be a harmonic coordinate patch and let $Z \subset U$ be the image of a smooth embedding of
the closed $k$-ball in $U$ for $k < n$. Let $N_\epsilon(Z)$ be the $\epsilon$-distance neighborhood of $Z$ with respect to
the Euclidean coordinates of $U$. If $\int_{N_\epsilon(Z)} \parallel g^{\mu \nu}  \parallel \cdot
|\log |g|| \dvol_g = o(\epsilon^2)$ as $\epsilon \rightarrow 0$ then
$dR$ vanishes on $\Int(Z)$. That is, for any Borel subset $W$ of $\Int(Z)$, we have $dR(W) = 0$.
\end{theorem}

The structure of the paper is the following.  In Section \ref{sect2} we clarify some points about
spinors and
Riemannian metrics.  In Section \ref{sect3} we prove Theorem \ref{thm1}.  In Section \ref{sect4} we prove Theorem \ref{thm2}.
Section \ref{sect5} has examples.  Section \ref{sect6} has a discussion of density bundles.
In Section \ref{sect7} we prove Theorem \ref{thm3}. In Section \ref{sect8} we prove Theorem \ref{thm4}. 
In Section \ref{sect9} we prove Theorem \ref{thm5}. More detailed descriptions are at the beginnings of the
sections.

\section{Background} \label{sect2}

In this section we recall the construction of spinor fields.  Since we will eventually
want to talk about smooth spinor fields
on a smooth manifold with a possibly singular Riemannian metric,
it is useful to disentangle spinor fields from Riemannian metrics as much as possible.

Let $M$ be a smooth $n$-dimensional manifold.  
Let $FGL(M)$ denote the frame bundle, a principal $\GL(n)$-bundle with projection map
$\alpha : FGL(M) \rightarrow M$.
There is a canonical $\R^n$-valued $1$-form $\tau$ on $FGL(M)$, the soldering form, defined as follows.  Suppose that $f \in FGL(M)$. Put $m = \alpha(f)$. Write
$f = (e_1, e_2, \ldots, e_n)$, with $e_i \in T_mM$. Given $v \in T_f FGL(M)$, write
$\alpha_*(v) = \sum_{i=1}^n v_i e_i$. Then $\tau(v) = (v_1, v_2, \ldots, v_n)$.

Putting $P = FGL(M)$, $G = \GL(n)$ and letting $T^VP$ denote the vertical tangent bundle
$\Ker(\alpha_*)$,
we have the following properties:
\begin{property} \label{property1}
\noindent
\begin{itemize} 
\item $\tau$ is $G$-equivariant,
\item $\tau$ vanishes on $T^VP$, and
\item For each $f \in P$, $\tau$ induces an
isomorphism $T_fP/T^V_fP \rightarrow \R^n$.
\end{itemize}
\end{property}

A Riemannian metric determines a reduction of the frame bundle from a principal $\GL(n)$-bundle to a principal $\OO(n)$-bundle $FO(M) \subset FGL(M)$, the bundle of orthonormal frames. 
Putting $P = FO(M)$ and $G = \OO(n)$, the pullback of $\tau$ to
$FO(M)$ satisfies Property \ref{property1}.
Suppose that $M$ is orientable.  Choosing an orientation amounts to giving a
reduction $FSO(M) \subset FO(M)$ of $FO(M)$ to a principal $\SO(n)$-bundle.
Putting $P = FSO(M)$ and $G = \SO(n)$, the further pullback of $\tau$ to
$FSO(M)$ satisfies Property \ref{property1}. Suppose in addition that $M$ has a spin structure.
This amounts to a fiberwise double cover $FSpin(M) \rightarrow FSO(M)$ with
connected fibers. Putting $P = FSpin(M)$ and $G = \Spin(n)$,  the further pullback of 
$\tau$ to $FSpin(M)$ still satisfies Property \ref{property1}.

Conversely, suppose that we have a topological reduction of $FGL(M)$ to a principal $\OO(n)$-bundle $\alpha : PO(M) \rightarrow M$.  (Such a reduction is unique up to isomorphism, since $\GL(n)/\OO(n)$ is contractible.)
If $M$ is orientable then a choice of orientation amounts to a reduction 
$PSO(M) \subset PO(M)$ of $PO(M)$ to a principal $\SO(n)$-bundle.
Suppose in addition that $M$ has a spin structure.
This amounts to a fiberwise double cover $PSpin(M) \rightarrow PSO(M)$ with
connected fibers.
Given a unitary representation $\rho : \Spin(n) \rightarrow \Aut(V)$, 
we can form the associated Hermitian vector bundle $SM = PSpin(M) \times_\rho V$, whose
sections are $V$-valued fields on $M$. The sections of $SM$ can be 
identified with $\Spin(n)$-equivariant maps from $PSpin(M)$ to $V$.
So far there is no Riemannian metric involved.

Going back to $PO(M)$ with projection map $\alpha : PO(M) \rightarrow M$,
suppose that we have an $\R^n$-valued $1$-form 
$\tau$ on $PO(M)$ that satisfies Property \ref{property1} with $G = \OO(n)$. Then we acquire a Riemannian metric
on $M$ as follows. Given $m \in M$, choose $f \in \alpha^{-1}(m)$. Given $w \in T_mM$, choose $\widehat{w} \in T_f PO(M)$ so that
$\alpha_* \widehat{w} = w$. Then the length of $w$ is defined to be
$|\tau(\widehat{w})|$. This is independent of the choices.

Thus a Riemannian metric on $M$ is equivalent to a $1$-form $\tau$ satisfying
Property \ref{property1}, living
on a fixed reduction $PO \rightarrow M$ of the frame bundle
$FGL(M)$. Changing the Riemannian metric amounts to changing $\tau$.

Putting $G = \SO(n)$, the pullback of $\tau$ to
$PSO(M)$ satisfies Property \ref{property1}. 
Putting $G = \Spin(n)$,  the further pullback of 
$\tau$ to $PSpin(M)$ still satisfies Property \ref{property1}.

In summary, a Riemannian spin $n$-manifold $M$ consists first of a
principal $\Spin(n)$-bundle $PSpin(M) \rightarrow M$ 
whose image under $\Spin(n) \rightarrow \SO(n)$ is a 
principal $\SO(n)$-bundle that is isomorphic to a reduction of the frame bundle $FGL(M)$.
Spinor fields and their pointwise inner products then exist on $M$, independent
of any Riemannian metric.  Adding a $\R^n$-valued $1$-form $\tau$ on $PSpin(M)$ satisfying Property \ref{property1}, relative to $G = \Spin(n)$, gives a Riemannian metric on $M$ and a Dirac operator on the spinor fields.

\begin{remark} \label{rem1}
As a historical note,
\'Elie Cartan's 1937 book ``The Theory of Spinors'' \cite{Cartan (1981)} ended with {\it THEOREM. With the geometric sense we have given to the word ``spinor'' it is
impossible to introduce fields of spinors into the classical Riemannian
technique}...  Cartan's conclusion was correct in that one cannot define spinor fields on arbitrary
Riemannian manifolds.  He didn't have the notion of a spin manifold. He did realize that spinor fields are closely
related to the existence of Riemannian metrics.
\end{remark}

\section{Scalar curvature measure} \label{sect3}

In this section we first consider smooth Riemannian metrics and write 
the formula for the scalar curvature measure in terms of local orthonormal frames.  
In Proposition \ref{prop1} we express the scalar curvature measure as the sum of a divergence
term and a term that is quadratic in the connection coefficients.
Using this decomposition, we pass to singular Riemannian metrics and
prove Theorem \ref{thm1} from the introduction. We then give
examples of the scalar curvature measure.

Let $M$ be a smooth $n$-manifold equipped with a Riemannian metric whose regularity will be
specified later.  As described in Section \ref{sect2}, we can form a smooth
principal $O(n)$-bundle $PO(M)$, independent
of the Riemannian metric.  We will only consider smooth gauge transformations of 
$PO(M)$.
The Riemannian metric is specified through the 
$\R^n$-valued soldering form $\tau$, or ``vierbein'' or ``vielbein'' in the physics terminology.  In local coordinates, we can write
$\tau^i =  \sum_\mu e_{\mu}^{\: \: i} dx^\mu$. The dual frame is
$e_i = \sum_\mu e_{i}^{\: \: \mu} \partial_\mu$, satisfying
$\sum_\mu e_{i}^{\: \: \mu} e_{\mu}^{\: \: j} = \delta_i^{\: \: j}$ and 
$\sum_i e_{\mu}^{\: \: i} \:  e_{i}^{\: \: \nu} = \delta_\mu^{\: \: \nu}$.
The Riemannian metric is $g_{\mu \nu} = \sum_i e_{\mu}^{\: \: i} e_{\nu}^{\: \: i}$.
The Riemannian volume density is $\dvol_g = |\tau^1 \wedge \ldots \wedge \tau^n| =
\sqrt{|g|} dx^1 dx^2 \ldots  dx^n$, where $|g| = \det \left( g_{\mu \nu} \right)$.
The connection $1$-forms, with respect to the local orthonormal frame, are
$\omega^i_{\: \: j} = \sum_\mu \omega^i_{\: \: j \mu} dx^\mu$ and can be computed by the 
metric compatibility condition
$\omega^i_{\: \: j} + \omega^j_{\: \: i} = 0$ and the torsion free condition
$d\tau^i + \sum_j \omega^i_{\: \: j} \wedge \tau^j = 0$.
We will use Greek letters for coordinate indices and Latin letters for orthonormal frame
indices.  We can move orthonormal frame indices up and down freely.
Hereafter we will use the
Einstein summation convention. We say that a tensor field on $M$ is $L^p$
with respect to $\dvol_g$ if its
components are $L^p$ with respect to $\dvol_g$ when written in local coordinates.

Suppose for the moment that the Riemannian metric is smooth.  The usual formula for
the scalar curvature measure $R \dvol_g$ in local coordinates, and using a local
orthonormal frame, is
\begin{align} \label{3.1}
R\dvol_g = & e_a^{\: \: \mu} e_b^{\: \: \nu} R^{ab}_{\: \: \: \: \mu \nu} \sqrt{|g|} dx^1 \ldots dx^n \\
= & e_a^{\: \: \mu} e_b^{\: \: \nu}
\left( \partial_\mu \omega^{ab}_{\: \: \: \: \: \nu} - \partial_\nu \omega^{ab}_{\: \: \: \: \mu} +
\omega^a_{\: \: c \mu} \omega^{cb}_{\: \: \: \:  \nu} - \omega^{a}_{\: \: c \nu} \omega^{cb}_{\: \: \: \:  \mu}
\right) \sqrt{|g|} dx^1 \ldots dx^n \notag \\
= & \sqrt{|g|} (e_a^{\: \: \mu} e_b^{\: \: \nu} - e_a^{\: \: \nu} e_b^{\: \: \mu})
\partial_\mu \omega^{ab}_{\: \: \: \: \nu} dx^1 \ldots dx^n + \notag \\
& \sqrt{|g|} e_a^{\: \: \mu} e_b^{\: \: \nu}
(
\omega^a_{\: \: c \mu} \omega^{cb}_{\: \: \: \:  \nu} - \omega^{a}_{\: \: c \nu} \omega^{cb}_{\: \: \: \:  \mu}
)
dx^1 \ldots dx^n. \notag
\end{align}

\begin{proposition} \label{prop1}
We can also write
\begin{align} \label{3.2}
R\dvol_g 
= & \partial_\mu \left( \sqrt{|g|} \left( e_a^{\: \: \mu} e_b^{\: \: \nu} - 
e_a^{\: \: \nu} e_b^{\: \: \mu}
\right) \omega^{ab}_{\: \: \: \: \nu} \right)
dx^1 \ldots dx^n - \\
& \sqrt{|g|} e_a^{\: \: \mu} e_b^{\: \: \nu}
(
\omega^a_{\: \: c \mu} \omega^{cb}_{\: \: \: \:  \nu} - \omega^{a}_{\: \: c \nu} \omega^{cb}_{\: \: \: \:  \mu}
)
dx^1 \ldots dx^n. \notag
\end{align}
Note the change in sign of the last term of (\ref{3.2}), as compared to (\ref{3.1}).
\end{proposition}
\begin{proof}
From (\ref{3.1}), we have
\begin{align} \label{3.3}
R\dvol_g 
= & \partial_\mu \left( \sqrt{|g|} \left( e_a^{\: \: \mu} e_b^{\: \: \nu} - 
e_a^{\: \: \nu} e_b^{\: \: \mu}
\right) \omega^{ab}_{\: \: \: \: \nu} \right)
dx^1 \ldots dx^n - \\
& \partial_\mu \left( \sqrt{|g|} \left( e_a^{\: \: \mu} e_b^{\: \: \nu} - 
e_a^{\: \: \nu} e_b^{\: \: \mu}
\right) \right) \omega^{ab}_{\: \: \: \: \nu} 
dx^1 \ldots dx^n + \notag \\
& \sqrt{|g|} (\omega^a_{\: \: ca} \omega^{cb}_{\: \: \: \:  b} - \omega^{a}_{\: \: c b} \omega^{cb}_{\: \: \: \:  a} )
dx^1 \ldots dx^n. \notag
\end{align}

Now
\begin{align} \label{3.4}
\frac{1}{\sqrt{|g|}} \partial_\mu \sqrt{|g|} = & \frac12 \frac{1}{{|g|}} \partial_\mu {|g|} =
\frac12 g^{\rho \sigma} \partial_\mu g_{\rho \sigma} \\
= &
\frac12 e_c^{\: \: \rho} e_c^{\: \: \sigma} \partial_\mu (e_\rho^{\: \: d} e_\sigma^{\: \: d}) =
e_c^{\: \: \rho} \partial_\mu e_\rho^{\: \: c} = - e_\rho^{\: \: c} \partial_\mu e_c^{\: \: \rho}. \notag
\end{align}
The torsion-free condition $d\tau^i + \omega^i_{\: \: j} \wedge \tau^j = 0$ gives
\begin{equation} \label{3.5}
\partial_\rho e_\sigma^{\: \: i} - \partial_\sigma e_\rho^{\: \: i} = - \omega^i_{\: \: j \rho} e_\sigma^{\: \: j} +
\omega^i_{\: \: j \sigma} e_\rho^{\: \: j},
\end{equation}
which implies
\begin{equation} \label{3.6}
\omega^i_{\: \: kl} - \omega^i_{\: \: lk} =
- e_k^{\: \: \rho} e_\sigma^{\: \: i} (\partial_\rho e_l^{\: \: \sigma}) + 
 e_l^{\: \: \sigma} e_\rho^{\: \: i} (\partial_\sigma e_k^{\: \: \rho}).
\end{equation}
In particular,
\begin{equation} \label{3.7}
\omega^i_{\: \: ki} =
- e_k^{\: \: \rho} e_\sigma^{\: \: i} (\partial_\rho e_i^{\: \: \sigma}) + 
\partial_\rho e_k^{\: \: \rho}.
\end{equation}

Then
\begin{align} \label{3.8}
& \frac{1}{\sqrt{|g|}} \partial_\mu \left( \sqrt{|g|} \left( e_a^{\: \: \mu} e_b^{\: \: \nu} - 
e_a^{\: \: \nu} e_b^{\: \: \mu}
\right) \right) = - e_\rho^{\: \: c} (\partial_\mu e_c^{\: \: \rho}) \left( e_a^{\: \: \mu} e_b^{\: \: \nu} - 
e_a^{\: \: \nu} e_b^{\: \: \mu}
\right) + \\
& (\partial_\mu e_a^{\: \: \mu}) e_b^{\: \: \nu} + e_a^{\: \: \mu} (\partial_\mu e_b^{\: \: \nu}) - 
 (\partial_\mu e_a^{\: \: \nu}) e_b^{\: \: \mu} - e_a^{\: \: \nu} (\partial_\mu e_b^{\: \: \mu}) \notag \\
& = \left[  \partial_\mu e_a^{\: \: \mu} - e_\rho^{\: \: c} (\partial_\mu e_c^{\: \: \rho}) e_a^{\: \: \mu} 
\right]
e_b^{\: \: \nu} - \left[ \partial_\mu e_b^{\: \: \mu} - e_\rho^{\: \: c} (\partial_\mu e_c^{\: \: \rho}) e_b^{\: \: \mu}  
\right]
e_a^{\: \: \nu} + \notag \\
& \left[ e_a^{\: \: \mu} (\partial_\mu e_b^{\: \: \nu}) - 
 (\partial_\mu e_a^{\: \: \nu}) e_b^{\: \: \mu} \right] \notag \\
 & = \omega^c_{\: \: ac} e_b^{\: \: \nu} - \omega^c_{\: \: bc} e_a^{\: \: \nu} -
 (\omega^c_{\: \: ab} - \omega^c_{\: \: ba}) e_c^{\: \: \nu}. \notag
\end{align}
Substituting (\ref{3.8}) into (\ref{3.3}), the proposition follows.
\end{proof}

We now assume that outside of a closed set ${\mathcal S} \subset M$ of measure zero,
in local coordinates $e_{\mu}^{\: \: i}$ is locally $L^1$, and its first partial 
derivatives as defined distributionally are measurable functions. Because we only consider
smooth gauge transformations, these conditions are well defined.

\begin{proposition} \label{prop2}
The formula (\ref{3.2}) gives a well defined distribution on $M$ provided that 
in coordinate charts, the following
expressions are $L^1$ with respect to $\dvol_{g}$:
\begin{enumerate}
\item $e_a^{\: \: \mu} e_b^{\: \: \nu} - e_b^{\: \: \mu} e_a^{\: \: \nu}$.
\item $\left( e_a^{\: \: \mu} e_b^{\: \: \nu} - 
e_a^{\: \: \nu} e_b^{\: \: \mu}
\right) \omega^{ab}_{\: \: \: \: \nu} $ and
$- e_a^{\: \: \mu} e_c^{\: \: \nu} \omega^{\: \: c}_{b \: \: \: \nu} +
e_b^{\: \: \mu} e_c^{\: \: \nu} \omega^{\: \: c}_{a \: \: \: \nu}
+ e_c^{\: \: \mu} e_b^{\: \: \nu} \omega^c_{\: \: a\nu} - e_c^{\: \: \mu} e_a^{\: \: \nu} \omega^c_{\: \: b\nu}$.
\item 
$e_a^{\: \: \mu} e_b^{\: \: \nu}
(
\omega^a_{\: \: c \mu} \omega^{cb}_{\: \: \: \:  \nu} - \omega^{a}_{\: \: c \nu} \omega^{cb}_{\: \: \: \:  \mu}
)$.
\end{enumerate}
\end{proposition}
\begin{proof}
Using a partition of unity, the distribution will be globally defined if it is well defined in coordinate
charts. With a fixed local orthonormal frame,
if 
$\left( e_a^{\: \: \mu} e_b^{\: \: \nu} - 
e_a^{\: \: \nu} e_b^{\: \: \mu}
\right) \omega^{ab}_{\: \: \: \: \nu} $
and 
$e_a^{\: \: \mu} e_b^{\: \: \nu}
(
\omega^a_{\: \: c \mu} \omega^{cb}_{\: \: \: \:  \nu} - \omega^{a}_{\: \: c \nu} \omega^{cb}_{\: \: \: \:  \mu}
)$
are
$L^1$ with respect to $\dvol_{g}$ then the right-hand side of (\ref{3.2}) is a well defined distribution in
the coordinate chart.  It remains to ensure that these conditions are invariant under gauge
transformations. 

Let $G$ be an $\OO(n)$-valued map which represents a gauge transformation. We write its components as
$G_a^{\: \: \overline{a}}$, and the components of $G^{-1}$ as $G_{\overline{a}}^{\: \: a}$.
The gauge transformation sends a vierbein to
$\overline{e}_\mu^{\: \: \overline{a}} = G_a^{\: \: \overline{a}} e_\mu^{\: \: a}$, an inverse
vierbein to 
$\overline{e}_{\overline{a}}^{\: \: \mu} = G_{\overline{a}}^{\: \: a} e_a^{\: \: \mu}$ 
and a 
connection form to
$\overline{\omega}^{\overline{a}}_{\: \: \overline{b} \mu} =
G_a^{\: \: \overline{a}} \omega^a_{\: \: b \mu} G_{\overline{b}}^{\: \: b} +
G_c^{\: \: \overline{a}} (\partial_\mu G_{\overline{b}}^{\: \: c} )$. 

Putting $\eta^a_{\: \: b \nu} = (\partial_\nu G_{\overline{d}}^{\: \: a}) G_b^{\: \: \overline{d}}$,
the effect of the gauge transformation on
$\left( e_a^{\: \: \mu} e_b^{\: \: \nu} - 
e_a^{\: \: \nu} e_b^{\: \: \mu}
\right) \omega^{ab}_{\: \: \: \: \nu}$ is to send it to
$
\left( e_a^{\: \: \mu} e_b^{\: \: \nu} - 
e_a^{\: \: \nu} e_b^{\: \: \mu}
\right) \left( \omega^{ab}_{\: \: \: \: \nu} + \eta^{ab}_{\: \: \: \: \nu} \right)$.
As $\eta$ is smooth, this is $L^1$ with respect to $\dvol_{g}$ since 
$e_a^{\: \: \mu} e_b^{\: \: \nu} - 
e_a^{\: \: \nu} e_b^{\: \: \mu}$ is.

By a similar argument, the gauge transformed version of 
$e_a^{\: \: \mu} e_b^{\: \: \nu}
(
\omega^a_{\: \: c \mu} \omega^{cb}_{\: \: \: \:  \nu} - \omega^{a}_{\: \: c \nu} \omega^{cb}_{\: \: \: \:  \mu}
)$
is $L^1$ with respect to $\dvol_g$ since 
$- e_a^{\: \: \mu} e_c^{\: \: \nu} \omega^{\: \: c}_{b \: \: \: \nu} +
e_b^{\: \: \mu} e_c^{\: \: \nu} \omega^{\: \: c}_{a \: \: \: \nu}
+ e_c^{\: \: \mu} e_b^{\: \: \nu} \omega^c_{\: \: a\nu} - e_c^{\: \: \mu} e_a^{\: \: \nu} \omega^c_{\: \: b\nu}$
is.
Finally, one checks that the gauge transformed version of 
$- e_a^{\: \: \mu} e_c^{\: \: \nu} \omega^{\: \: c}_{b \: \: \: \nu} +
e_b^{\: \: \mu} e_c^{\: \: \nu} \omega^{\: \: c}_{a \: \: \: \nu}
+ e_c^{\: \: \mu} e_b^{\: \: \nu} \omega^c_{\: \: a\nu} - e_c^{\: \: \mu} e_a^{\: \: \nu} \omega^c_{\: \: b\nu}$
is $L^1$ with respect to $\dvol_g$ since $e_a^{\: \: \mu} e_b^{\: \: \nu} - 
e_a^{\: \: \nu} e_b^{\: \: \mu}$ is.
\end{proof}

\begin{remark} \label{rem2}
We can write the expressions in Proposition \ref{prop2} more succinctly as
\begin{enumerate}
\item $e_a^{\: \: \mu} e_b^{\: \: \nu} - e_b^{\: \: \mu} e_a^{\: \: \nu}$.
\item $e_a^{\: \: \mu} \omega^{ab}_{\: \: \: \: b} $ and
$- e_a^{\: \: \mu} \omega^{\: \: c}_{b \: \: \: c} +
e_b^{\: \: \mu} \omega^{\: \: c}_{a \: \: \: c}
+ e_c^{\: \: \mu} \omega^c_{\: \: ab} - e_c^{\: \: \mu} \omega^c_{\: \: ba}$.
\item 
$\omega^a_{\: \: ca} \omega^{cb}_{\: \: \: \:  b} - \omega^{a}_{\: \: cb} \omega^{cb}_{\: \: \: \:  a}$.
\end{enumerate}
\end{remark}

\begin{remark} \label{rem3}
The condition that $e_a^{\: \: \mu} e_b^{\: \: \nu} - e_b^{\: \: \mu} e_a^{\: \: \nu}$ be $L^1$ has the invariant global
characterization that
any smooth antisymmetric bivector field, i.e. any smooth section of
$\Lambda^2 TM$, is $L^1$ with respect
to $\dvol_g$. The geometric meaning of the other conditions in Proposition \ref{prop2} is less clear.
\end{remark}

\begin{remark} \label{rem4} It would be more natural in some ways to assume that $M$ is a $C^{1,1}$-manifold
and deal with Lipschitz-regular gauge transformations.  Since a $C^{1,1}$-manifold has an underlying
$C^\infty$-structure, unique up to diffeomorphism, we would not gain any generality and so in this paper we restrict to 
smooth manifolds.
\end{remark}

\begin{definition} \label{def1}
When the assumptions of Proposition \ref{prop2} hold, we call the right-hand side of (\ref{3.2}) the scalar curvature
distribution $dR$. If it is a measure, i.e. if $\int_M f \: dR$ extends to a continuous
linear function of $f \in C_c(M)$, then we call it the scalar curvature measure.
\end{definition}

Here if $M$ is noncompact then $C_c(M)$ has the LF-topology.
If the scalar curvature distribution $dR$ is nonnegative, in the sense that
$\int_M f \: dR \ge 0$ for all nonnegative $f \in C^\infty_c(M)$, then it is a measure.

\begin{example} \label{ex1}
Suppose that $\{e_\mu^{\: \: a} \}$ and $\{e_a^{\: \: \mu}\}$ are bounded, and $\{ \omega^a_{\: \: b \mu} \}$ 
are $L^2$ with respect to the Euclidean volume form,
in local coordinate charts.
Then $|g| = \det(g_{\mu \nu}) = \det({e_\mu^{\: \: a}} e_\nu^{\: \: a})$ is also bounded and
the conditions of Proposition \ref{prop2} are satisfied. 
\end{example}

\begin{example} \label{ex2}
Suppose that $M$ is a smooth surface. 
We consider Riemannian metrics $g$ that can be written in local isothermal coordinates as
$g = e^{2 \sigma} (dx^2 + dy^2)$, where $\sigma$
lies in $W_{loc}^{1,1}$ with respect to the Euclidean structure. 
Put $\tau^1 = e^\sigma dx$ and $\tau^2 = e^\sigma dy$. Then
$e_x^{\: \: 1} = e_y^{\: \: 2} = e^\sigma$, $e_x^{\: \: 2} = e_y^{\: \: 1} = 0$,
$\sqrt{|g|} = e^{2 \sigma}$ and $\omega^1_{\: \: 2} = \phi_y dx - \phi_x dy$. 
The assumptions of Proposition \ref{prop2} are satisfied and
one finds that $dR = - 2 (\triangle_{\R^2} \sigma) dx dy$, as one would expect.
It is a measure if, for example, $(M, g)$ has Alexandrov curvature bounded below. 

Under these assumptions, the connection form may not be $L^2$, in the sense that
$\omega^a_{\: \: bc}$ may not be $L^2$ with respect to $\dvol_g$. However,
$\omega^a_{\: \: ca} \omega^{cb}_{\: \: \: \:  b} - \omega^{a}_{\: \: cb} \omega^{cb}_{\: \: \: \:  a}$
is $L^1$ because it vanishes identically in two dimensions.
\end{example}

\begin{example} \label{ex3}
Consider the conical metric $g = |x|^{-2c} \left( (dx^1)^2 + \ldots + (dx^n)^2 \right)$ on
$\R^n$. Here $c < 1$. 
Let $h$ be the standard round metric on $S^{n-1}$. Putting $s = \frac{|x|^{1-c}}{1-c}$, the
metric $g$ can be written as $g = ds^2 + (1-c)^2 s^2 h$, which is biLipschitz to the Euclidean metric
$ds^2 + s^2 h$. Let $\{ \widehat{\tau}^i \}_{i=1}^{n-1}$ be a local orthonormal frame around a point
on $(S^{n-1}, h)$. Then $\{ ds, (1-c) s \widehat{\tau}^1, \ldots, (1-c) s \widehat{\tau}^{n-1} \}$
gives local orthonormal frames for $g$. One finds that as $s \rightarrow 0$, the 
connection coefficients $\omega^a_{\: \: b \mu}$ are asymptotic to $s^{-1}$. As $\dvol_g = \const
s^{n-1} \: ds \: \dvol_{S^{n-1}}$, the connection coefficients are $L^2$ if $n > 2$, and then Example
\ref{ex1} applies.  Hence $dR$ is well defined if $n > 2$, and equals $\const s^{-2} \dvol_g$.

On the other hand, if $n = 2$ then the connection coefficients are not $L^2$. Nevertheless, from Example \ref{ex2},
the measure $dR$ is well defined and equals $4 \pi c \delta_0$.
\end{example}

\section{Integrated Lichnerowicz formula} \label{sect4}

In this section we use the Lichnerowicz formula to define the scalar curvature measure. We give examples and
prove Theorem \ref{thm2} from the introduction.

We continue the setup of Section \ref{sect3}, except that we now assume that $M$ is spin.
Let $SM$ denote the spinor bundle.
Let $\{\gamma^i \}_{i=1}^n$ denote the Clifford matrices,
satisfying $\gamma^i \gamma^j + \gamma^j \gamma^i = 2 \delta^{ij}$.
If $\psi$ is a smooth
spinor field then the Dirac operator $D$ acts on it by
$D \psi =  - \sqrt{-1} \gamma^i e_i^{\: \: \mu} \left( \partial_\mu \psi +
\frac18 \omega_{jk \mu} [\gamma^j, \gamma^k] \psi \right)$.

If $g$ is smooth then the Lichnerowicz formula says that 
\begin{equation} \label{4.1}
\int_M \left( \langle D\eta, D\psi \rangle -  \langle \nabla \eta, \nabla \psi \rangle \right) \: \dvol_g = \frac14 \int_M \langle \eta, \psi \rangle \: R \: \dvol_g,
\end{equation}
where $\eta$ and $\psi$ run over $C_c^\infty(SM)$.
If $g$ is not smooth then it is possible that the left-hand side of (\ref{4.1}) 
makes sense even if $\int_M \langle D\eta, D\psi \rangle \: \dvol_g$ and $\int_M \langle \nabla \eta, \nabla \psi \rangle \: \dvol_g$ are
not individually finite.

\begin{example} \label{ex4}
Suppose that $M$ is a smooth compact spin surface. Let $g_0$ be a smooth
Riemannian metric on $M$. We consider Riemannian metrics $g = e^{2 \sigma} g_0$, where $\sigma
\in W^{1,1}(M)$. Let $\{ \tau_0^1, \tau_0^2 \}$ be a local orthonormal coframe for
$(M, g_0)$, with dual orthonormal frame $\{ e^0_1, e^0_2 \}$.
Put $\tau^i = e^\sigma \tau^i_0$. One finds
\begin{align} \label{4.2}
\nabla_{e_1} \psi = & e^{- \sigma} \left( \nabla_{e^0_1} \psi + \frac12 (e^0_2\sigma) \gamma^1 \gamma^2 \psi \right), \\
\nabla_{e_2} \psi = & e^{-\sigma} \left( \nabla_{e^0_2} \psi - \frac12 (e^0_1\sigma)
\gamma^1 \gamma^2 \psi \right), \notag \\
D\psi = & e^{- \sigma} D_0\psi - \frac12 \sqrt{-1} e^{- \sigma} \left(
\gamma^1 e^0_1 \sigma + \gamma^2 e^0_2 \sigma
\right) \psi.\notag
\end{align}
Some calculation gives
\begin{equation} \label{4.3}
\langle D\eta, D\psi \rangle -  \langle \nabla \eta, \nabla \psi \rangle = 
e^{- 2 \sigma} \left( \langle D_0 \eta, D_0 \psi \rangle -  \langle \nabla_0 \eta, \nabla_0 \psi \rangle \right) + \frac12 e^{- 2 \sigma} \langle \nabla_0 \sigma, \nabla_0 \langle \eta, \psi \rangle \rangle_{g_0}
\end{equation}
and so
\begin{equation} \label{4.4}
\int_M \left( \langle D\eta, D\psi \rangle -  \langle \nabla \eta, \nabla \psi \rangle 
\right) \dvol_g = \frac14 \int_M \langle \eta, \psi \rangle dR_0 - \frac12 \int_M
\langle \nabla_0 \sigma, \nabla_0 \langle \eta, \psi \rangle \rangle_{g_0}
\dvol_{g_0}.
\end{equation}
We note that the right-hand side of (\ref{4.4}) makes sense under our assumption that
$\sigma$ has first partial derivatives which are $L^1$ on $M$. Then
$(\triangle_0 \sigma) \dvol_{g_0}$ makes sense as a distribution and
\begin{equation} \label{4.5}
\int_M \left( \langle D\eta, D\psi \rangle -  \langle \nabla \eta, \nabla \psi \rangle 
\right) \dvol_g = \frac14 \int_M \langle \eta, \psi \rangle \left( dR_0 - 2
(\triangle_0 \sigma) \dvol_{g_0} \right).
\end{equation}
Hence
$dR = dR_0 - 2
(\triangle_0 \sigma) \dvol_{g_0}$, which is consistent with Example \ref{ex2}.

In order for 
$\int_M \langle D\eta, D\psi \rangle \dvol_g$ and $\int_M \langle \nabla \eta, \nabla \psi \rangle 
\dvol_g$ to make sense individually, for all $\eta, \psi \in C^\infty(SM)$, we need to have $\sigma \in W^{1,2}(M)$.
Because of cancellations in $\langle D\eta, D\psi \rangle -  \langle \nabla \eta, \nabla \psi \rangle$ this can be
relaxed to $\sigma \in W^{1,1}(M)$ when defining $dR$.
\end{example}

\begin{example} \label{ex5}
To amplify the last point,
consider the conical metric $|x|^{-2c} \left( (dx^1)^2 + \ldots + (dx^n)^2 \right)$ on
$\R^n$. Here $c < 1$. The link of the cone is $S^{n-1}$ with a Riemannian metric that is
$(1-c)^2$ times the standard round metric. Putting $e^{2 \sigma} = |x|^{-2c}$ and
letting $D_{flat}$ denote that Dirac operator on flat $\R^n$, the formula for the
Dirac operator under a conformal change of metric 
\cite[Section 1.4]{Hitchin (1974)}, \cite[Proposition 2]{Lott (1986)}
gives
\begin{align} \label{4.6}
D \psi = &  e^{- \: (n+1) \sigma/2} D_{flat} \left( e^{(n-1) \sigma/2} \psi\right) =
|x|^{c(n+1)/2} D_{flat} \left( |x|^{-c(n-1)/2} \psi\right) = \\
& 
|x|^{c} D_{flat} \psi - \sqrt{-1} \gamma^i |x|^{c(n+1)/2} 
\left( \partial_i |x|^{-c(n-1)/2} \right) \psi \notag
\end{align}
for $\psi \in C^\infty_c(SM)$.
If $c \neq 0$ then because of the last term in (\ref{4.6}), $|D \psi|$ will generally be of order $|x|^{c-1}$ as $x \rightarrow 0$ and
$|D \psi|$ will be in $L^p(B_\epsilon, \sqrt{|g|} dx^1 \ldots dx^n)$ if and only if
$\int_0^\epsilon r^{p(c-1)} r^{-nc} r^{n-1} dr < \infty$, i.e. 
$\int_0^\epsilon r^{(1-c)(n-p)} r^{-1} dr < \infty$, i.e. $p<n$. In particular, $D$
will map smooth compactly supported spinors to $L^2$-spinors if and only if $n > 2$.
\end{example}

Let $\Alt$ denote antisymmetrization of a tensor, with a normalization factor of $\frac{1}{p!}$ when
$\Alt$ acts on a $p$-tensor.

\begin{proposition} \label{prop3}
The left-hand side of (\ref{4.1}) is well defined for all $\eta, \psi \in C^\infty_c(SM)$ if and only if in local 
coordinates, the following expressions are $L^1$ with respect to $\dvol_{g}$:
\begin{enumerate}
\item $e_c^{\: \: \mu} e_d^{\: \: \nu} - e_d^{\: \: \mu} e_c^{\: \: \nu}$.
\item $e_c^{\: \: \mu} \omega^{cd}_{\: \: \: \: d}$,
$- e_a^{\: \: \mu} \omega_{b\: \: d}^{\: \: d}
+ e_b^{\: \: \mu} \omega_{a\: \: d}^{\: \: d} + e_c^{\: \: \mu}\omega^c_{\: \: ab} - e_c^{\: \: \mu}\omega^c_{\: \: ba}$ and
$\Alt_{cdef} (e_c^{\: \: \mu} \omega_{efd})$.
\item $\omega^a_{\: \: ca} \omega^{cb}_{\: \: \: \:  b} - \omega^{a}_{\: \: cb} \omega^{cb}_{\: \: \: \:  a}$ and
$\Alt_{mnpq} \left(\omega_{m \: \: r}^{\: \: r} \omega_{pqn} - 
\omega_{pqr} \omega^r_{\: \: nm} + \omega_{mrp} \omega^r_{\: \: nq}
\right)$.
\end{enumerate}
\end{proposition}
\begin{proof}
As a pointwise statement, we have    
\begin{align} \label{4.7}
& \langle D\eta, D\psi \rangle -  \langle \nabla \eta, \nabla \psi \rangle = 
\langle \gamma^c \nabla_c \eta, \gamma^d \nabla_d \psi \rangle -
\langle  \nabla_c \eta, \nabla_c \psi \rangle = \\
& \langle \nabla_c \eta, \gamma^c \gamma^d \nabla_d \psi \rangle -
\langle  \nabla_c \eta, \nabla_c \psi \rangle = 
\frac12 \langle \nabla_c \eta, [\gamma^c, \gamma^d] \nabla_d \psi \rangle = \notag \\
& \frac12 e_c^{\: \: \mu} e_d^{\: \: \nu} \langle \nabla_\mu \eta, [\gamma^c, \gamma^d] \nabla_\nu \psi \rangle = \notag \\
& \frac12 e_c^{\: \: \mu} e_d^{\: \: \nu} \langle 
\partial_\mu \eta + \frac18 \omega_{ab\mu} [\gamma^a, \gamma^b] \eta,
[\gamma^c, \gamma^d] ( \partial_\nu \psi + \frac18 \omega_{ef\nu} [\gamma^e, \gamma^f] \psi ) \rangle. \notag
\end{align}
Following the logic of \cite[Proof of Proposition 3.3]{Lott (2016)}, for the left-hand side of (\ref{4.7}) to be well defined for all
$\eta, \psi \in C_c^\infty(SM)$, we need the following matrix-valued expressions to be $L^1$:
\begin{enumerate}
\item $e_c^{\: \: \mu} e_d^{\: \: \nu} [\gamma^c, \gamma^d]$,
\item $e_c^{\: \: \mu} e_d^{\: \: \nu} \omega_{ef\nu} [\gamma^c, \gamma^d] [\gamma^e, \gamma^f]$ and
\item $e_c^{\: \: \mu} e_d^{\: \: \nu} \omega_{ab\mu} \omega_{ef\nu} [\gamma^a, \gamma^b] 
[\gamma^c, \gamma^d] [\gamma^e, \gamma^f]$.
\end{enumerate}

For notation, we will write $\gamma^{[i_1 i_2 \ldots i_k]} = \Alt^{i_1 i_2 \ldots i_k} (\gamma^{i_1} 
\gamma^{i_2} \ldots \gamma^{i_k})$.

For (1), we write $e_c^{\: \: \mu} e_d^{\: \: \nu} [\gamma^c, \gamma^d] = 
\frac12 (e_c^{\: \: \mu} e_d^{\: \: \nu} - e_d^{\: \: \mu} e_c^{\: \: \nu}) [\gamma^c, \gamma^d]$.

For (2), we start with the identity
\begin{equation} \label{4.8}
[\gamma^c,\gamma^d][\gamma^e,\gamma^f]=4\gamma^{[cdef]}+4( \delta^{de} \gamma^{[cf]}-\delta^{df}
\gamma^{[ce]}-\delta^{ce} \gamma^{[df]}+\delta^{cf} \gamma^{[de]})-4(\delta^{ce} \delta^{fd}
-\delta^{cf} \delta^{ed})1.
\end{equation}
Then $e_c^{\: \: \mu} e_d^{\: \: \nu} \omega_{ef\nu} [\gamma^c, \gamma^d] [\gamma^e, \gamma^f] = 
e_c^{\: \: \mu} \omega_{efd} [\gamma^c, \gamma^d] [\gamma^e, \gamma^f]$ becomes
\begin{equation} \label{4.9}
e_c^{\: \: \mu} \left( 4\omega_{efd} \gamma^{[cdef]}+4(\delta_a^c \delta^{de} \omega_{ebd}
- \delta_b^c \delta^{de} \omega_{ead} + \omega_{cab} - \omega_{cba})
\gamma^{[ab]}+8
\delta^{cf} \delta^{de} \omega_{efd} 1
\right).
\end{equation}
If (\ref{4.9}) is to be $L^1$ then the coefficients of the individual basis elements have to be $L^1$.

For (3), we start with the identity (computer generated and verified)
\begin{equation} \label{4.10}
[\gamma^a,\gamma^b][\gamma^c,\gamma^d][\gamma^e,\gamma^f]
=8\big(\mathcal G_6+\mathcal G_4+\mathcal G_2+\mathcal G_0\big),
\end{equation}
where
$\mathcal G_6=\gamma^{[abcdef]}$,
\begin{align} \label{4.11}
\mathcal G_4
 =& \delta^{bc}\gamma^{[adef]}-\delta^{bd}\gamma^{[acef]}-\delta^{ac}\gamma^{[bdef]}+\delta^{ad}\gamma^{[bcef]} +\\
&\delta^{de}\gamma^{[cfab]}-\delta^{df}\gamma^{[ceab]}-\delta^{ce}\gamma^{[dfab]}+\delta^{cf}\gamma^{[deab]} + \notag \\
&\delta^{af}\gamma^{[becd]}-\delta^{bf}\gamma^{[aecd]}-\delta^{ae}\gamma^{[bfcd]}+\delta^{be}\gamma^{[afcd]}\;, \notag
\end{align}

\begin{align} \label{4.12}
\mathcal G_2
= & (\delta^{ed}\delta^{fa}-\delta^{ea}\delta^{fd})\,\gamma^{[bc]}
+(\delta^{ea}\delta^{fc}-\delta^{ec}\delta^{fa})\,\gamma^{[bd]}
+(\delta^{eb}\delta^{fd}-\delta^{ed}\delta^{fb})\,\gamma^{[ac]}
+ \\
& (\delta^{ec}\delta^{fb}-\delta^{eb}\delta^{fc})\,\gamma^{[ad]} 
+(\delta^{ed}\delta^{fc}-\delta^{ec}\delta^{fd})\,\gamma^{[ab]}
+(\delta^{eb}\delta^{fa}-\delta^{ea}\delta^{fb})\,\gamma^{[cd]}
+ \notag \\
& (\delta^{bc}\delta^{de}-\delta^{bd}\delta^{ce})\,\gamma^{[af]}
+(-\delta^{bc}\delta^{df}+\delta^{bd}\delta^{cf})\,\gamma^{[ae]}
+(\delta^{ac}\delta^{df}-\delta^{ad}\delta^{cf})\,\gamma^{[be]}
+ \notag \\
& (-\delta^{ac}\delta^{de}+\delta^{ad}\delta^{ce})\,\gamma^{[bf]}
+(\delta^{bc}\delta^{af}-\delta^{ac}\delta^{bf})\,\gamma^{[de]}
+(-\delta^{bc}\delta^{ae}+\delta^{ac}\delta^{be})\,\gamma^{[df]}
+ \notag \\
& (-\delta^{bd}\delta^{af}+\delta^{ad}\delta^{bf})\,\gamma^{[ce]}
+(\delta^{bd}\delta^{ae}-\delta^{ad}\delta^{be})\,\gamma^{[cf]}
+(-\delta^{ac}\delta^{bd}+\delta^{ad}\delta^{bc})\,\gamma^{[ef]}\; \notag
\end{align}
and
\begin{align} \label{4.13}
\mathcal G_0= &
-\delta^{bc}\delta^{ae}\delta^{df}
+\delta^{bc}\delta^{af}\delta^{de}
+\delta^{bd}\delta^{ae}\delta^{cf}
-\delta^{bd}\delta^{af}\delta^{ce} + \\
& \delta^{ac}\delta^{be}\delta^{df}
-\delta^{ac}\delta^{bf}\delta^{de}
-\delta^{ad}\delta^{be}\delta^{cf}
+\delta^{ad}\delta^{bf}\delta^{ce}. \notag
\end{align}
After multiplying by $\omega_{abc} \omega_{efd}$ and contracting indices, the 
contributions of $G_6$ and $G_2$ vanish, and we obtain
\begin{align} \label{4.14}
& \omega_{abc} \omega_{efd} [\gamma^a, \gamma^b] 
[\gamma^c, \gamma^d] [\gamma^e, \gamma^f] = \\
&
768 \mathrm{Alt}_{mnpq}
\big(\omega_{rmm} \omega_{pqn}- \omega_{pqc} \omega_{cnm} + \omega_{mbp} \omega_{bnq}\big) \gamma^{[mnpq]}
+ 32( \omega_{abc} \omega_{acb}-\omega_{arr} \omega_{ass}). \notag
\end{align}
The proposition follows.
\end{proof}

\begin{proposition} \label{prop4}
If the left-hand side of (\ref{4.1}) makes sense for all $\eta, \psi \in C_c^\infty(SM)$
then the scalar distribution $dR$ of Definition \ref{def1} is well defined and 
\begin{equation} \label{4.15}
\int_M \left( \langle D\eta, D\psi \rangle -  \langle \nabla \eta, \nabla \psi \rangle \right) \: \dvol_g = \frac14 \int_M \langle \eta, \psi \rangle \: dR.
\end{equation}
\end{proposition}
\begin{proof}
The relevant part of (\ref{4.7})
is
\begin{align} \label{4.16}
&\langle D\eta, D\psi \rangle -  \langle \nabla \eta, \nabla \psi \rangle = 
\frac12 e_c^{\: \: \mu} e_d^{\: \: \nu} \langle \nabla_\mu \eta, [\gamma^c, \gamma^d] \nabla_\nu \psi \rangle = \\
& \frac12 e_c^{\: \: \mu} e_d^{\: \: \nu} \langle 
\partial_\mu \eta + \frac18 \omega_{ab\mu} [\gamma^a, \gamma^b] \eta,
[\gamma^c, \gamma^d] ( \partial_\nu \psi + \frac18 \omega_{ef\nu} [\gamma^e, \gamma^f] \psi ). \rangle \notag
\end{align}
In order to derive the right-hand side of (\ref{4.15}), we have to do
an integration by parts to move the $\nabla_\mu$-derivative off of $\nabla_\mu \eta$.
In so doing we obtain a divergence term
that, in local coordinates, is 
\begin{equation} \label{4.17}
\partial_\mu \left( \frac12 \sqrt{|g|} e_c^{\: \: \mu} e_d^{\: \: \nu} \langle 
\eta,
[\gamma^c, \gamma^d] ( \partial_\nu \psi + \frac18 \omega_{ef\nu} [\gamma^e, \gamma^f] \psi ) \rangle\right)
\end{equation}
From Proposition \ref{prop3}, $\sqrt{|g|} e_c^{\: \: \mu} e_d^{\: \: \nu} [\gamma^c, \gamma^d]$ and
$\sqrt{|g|} e_c^{\: \: \mu} e_d^{\: \: \nu} \omega_{ef\nu} [\gamma^c, \gamma^d] [\gamma^e, \gamma^f]$
are $L^1$ with respect to the Euclidean volume form. Then using a partition of unity, the integration by parts
is justified and we obtain 
\begin{equation} \label{4.18}
\int_M \left( \langle D\eta, D\psi \rangle -  \langle \nabla \eta, \nabla \psi \rangle \right) \: \dvol_g =
\frac12 \int_M e_c^{\: \: \mu} e_d^{\: \: \nu} \langle \eta, [\gamma^c, \gamma^d] \nabla_\mu  \nabla_\nu \psi \rangle
\dvol_g,
\end{equation}
where the right-hand side is interpreted as the pairing of a distribution with the smooth spinor $\eta$. 
At this point, standard manipulations show that the right-hand side of (\ref{4.18}) equals 
$\frac14 \int_M \langle \eta, \psi \rangle \: dR$, where $dR$ is the distribution given by the right-hand side of
(\ref{3.2}).
\end{proof}

\section{Examples of scalar curvature measure} \label{sect5}

In this section we compute the scalar curvature measure $dR$ of Section \ref{sect3} in two geometric situations.
The first one, when two Riemannian manifolds are glued together is in Subsection \ref{subsect5.1}.
The second one, a family of conical metrics, is in Subsection \ref{subsect5.2}. 

The computations here
are analogous to those for the Ricci curvature in \cite[Section 4]{Lott (2016)}.

\subsection{Gluing} \label{subsect5.1}
Let $M_1$ and $M_2$ be Riemannian manifolds with boundary. Let $H_1$ and $H_2$ denote the
mean curvatures of $\partial M_1$ and $\partial M_2$, respectively, as computed using the inward
pointing normals.  Our convention is that the mean curvature of the boundary of the unit ball in $\R^n$ is $n-1$.

Let $\phi : \partial M_1 \rightarrow \partial M_2$ be an isometric diffeomorphism.  Using the
local product structure near $\partial M_1$ (resp. $\partial M_2$) coming from the normal exponential
map, the result $M = M_1 \cup_\phi M_2$ of gluing $M_1$ to $M_2$ acquires a product structure.
It also acquires a $C^0$-Riemannian metric. Let $X \subset M$ denote the gluing locus.

\begin{proposition} \label{prop5}
The measure $dR$ is well defined. For $f \in C_c(M)$, we have
\begin{equation} \label{5.1}
\int_M f \: dR = \int_M f R \dvol_M + 2 \int_X f (H_1 + H_2) \dvol_X.
\end{equation}
\end{proposition}
\begin{proof}
It is clear that $dR$ equals $R \dvol_M$ on $M-X$. The only issue is to understand the singularity at $X$,\
which can only come from the first term on the right-hand side of (\ref{3.2}). If $x \in X$
and $\{x^\mu\}_{\mu=1}^{n-1}$
are normal coordinates on $X$ around $x$
then we supplement them by the coordinate $x^0$ coming from the normal
exponential map of $X$, to get local coordinates on $M$ in a neighborhood of $x$. 
Let $\{e_i\}_{i=1}^{n-1}$ be a local orthonormal frame for $X$, in a neighborhood of $x$,
so that $e_i(x) = \partial_{x^i}$. We
can parallel transport the frame along normal geodesics and add $e_0 = \partial_{x_0}$, to get a local
orthonormal frame for $M$ in a neighborhood of $x$. Looking at (\ref{3.2}), the possible singularity at $x$
comes from the $\partial_0$-term in the divergence expression and equals $\delta_0(x^0)$ times the
jump in $2 \omega^{0b}_{\: \: \: \: b}$ when going across $X$. As one approaches $x$ in a normal direction from
the interior of $M_i$, the limit of $\omega^{0b}_{\: \: \: \: b}$ is $\pm H_i(x)$, 
with the sign depending on whether or not $\partial_0$ is the inward normal for $M_i$. 
The proposition follows.
\end{proof}

\begin{remark} \label{rem5}
Based on Proposition \ref{prop5}, it is natural to say that the scalar curvature measure of a smooth Riemannian
manifold-with-boundary consists of the scalar curvature measure of the interior plus twice the mean curvature times the delta measure of the boundary.  Then the total scalar curvature becomes the Gibbons-Hawking-York functional.
\end{remark}

\subsection{Family of cones} \label{subsect5.2}

We now consider a family of cones.
Let $\pi : M \rightarrow B$ be an $n$-dimensional 
real vector bundle over a Riemannian manifold $B$. 
Given $b \in B$, we write $M_b = \pi^{-1}(b)$.
Let $h$ be a Euclidean
inner product on $M$ and let $D$ be an $h$-compatible connection.
There is a natural Riemannian metric $g_0$ on $M$ with $\pi : M \rightarrow B$
being a Riemannian submersion, so that the restrictions of $g_0$ to fibers
are specified by $h$, and with horizontal subspaces coming from 
$D$.
Let $Z$ be the zero section of the vector bundle. Given $c < 1$, 
let $g$ be the Riemannian metric on
$M - Z$ obtained from $g_0$, at $m \in M-Z$, by multiplying the fiberwise 
component of $g_0$ by $h(m,m)^{- c}$.

\begin{proposition} \label{prop6}
The measure $dR$ is well defined.  If $n>2$ then for $f \in C_c(M)$, we have
$\int_M f \: dR = \int_M f R \dvol_M$.  If $n = 2$ then for $f \in C_c(M)$, we have
\begin{equation} \label{5.2}
\int_M f \: dR = \int_M f R \dvol_M + 4 \pi c \int_Z f \dvol_Z.
\end{equation}
\end{proposition}
\begin{proof}
We can choose local coordinates $\{ x^{\mu}, x^{\widehat{\mu}} \}$ for $M$ so that the
coordinates $\{x^\mu\}$ pull back from $B$ and the coordinates $\{x^{\widehat{\mu}}\}$ restrict
to the fibers as linear orthogonal coordinates with respect to $h$. We take
a local orthonormal coframe $\{\tau^i, \tau^{\widehat{i}} \}$ so that $\{\tau^i \}$
pulls back from $B$, and the components $\{e_{\widehat{i}} \}$ of the dual frame satisfy
$e_{\widehat{i}} = |\widehat{x}|^{c} \partial_{x^i}$. We can write 
\begin{equation} \label{5.3}
\tau^{\widehat{i}} = |\widehat{x}|^{-c} \left( dx^{\widehat{i}} + \sum_{\beta, \widehat{j}}
C^{\widehat{i}}_{\: \: \widehat{j} \mu} x^{\widehat{j}} dx^\mu \right),
\end{equation}
where $\{ C^{\widehat{i}}_{\: \: \widehat{j} \mu}\}$ are the Christoffel symbols for the
connection $D$. If $F$ denote the curvature of $D$ then the components $\omega^i_{\: \: jk}$ 
of the connection $1$-form of $(M, g)$ pull back 
from $B$, and the other nonzero components are
\begin{itemize}
\item  $\omega^{\widehat{i}}_{\: \: \widehat{j} \mu} = C^{\widehat{i}}_{\: \: \widehat{j} \mu}$,
\item $\omega^{\widehat{i}}_{\: \: kl} = \frac12 |\widehat{x}|^{-c} F^{\widehat{i}}_{\: \: \widehat{j} kl}$ and
\item $\omega^{\widehat{i}}_{\: \: \widehat{j} \widehat{k}} = \frac12 c |\widehat{x}|^{-c}
\left( \delta_{{j}}^{{i}} x^{\widehat{k}} - \delta_{{k}}^{{i}} x^{\widehat{j}}
\right)$.
\end{itemize}

Away from $Z$, we clearly have $dR = R_M \: \dvol_M$. The possible singularities in $dR$ come from 
the divergence term in (\ref{3.2}), which can be written as
\begin{equation} \label{5.4}
2 \left( \nabla_{i} \omega^{ij}_{\: \: \: \: j} + \nabla_{i} \omega^{i\widehat{j}}_{\: \: \: \: \widehat{j}}
+ \nabla_{\widehat{i}} \omega^{\widehat{i}j}_{\: \: \: \: j} + \nabla_{\widehat{i}} \omega^{\widehat{i}\widehat{j}}_{\: \: \: \: \widehat{j}}
\right) \dvol_M.
\end{equation}
Looking at the components of $\omega$, the only term that can contribute to a singularity is the last one,
$2 \nabla_{\widehat{i}} \omega^{\widehat{i}\widehat{j}}_{\: \: \: \: \widehat{j}} \dvol_M$. 
This is the contribution coming from a single fiberwise cone, which is given in Example \ref{ex3}.
The proposition follows.
    \end{proof}

\section{Density bundles} \label{sect6}

In this section we recall some facts about density bundles.

Let $M$ be a smooth manifold.  Given $\beta \in \R$, the $\beta$-density bundle ${\mathcal D}^\beta$ has local sections
$f (dx^1 \ldots dx^n)^\beta$.  A Riemannian metric $g$ induces an inner product on ${\mathcal D}^\beta$  by
\begin{equation} \label{6.1}
\langle f (dx^1 \ldots dx^n)^\beta, f^\prime (dx^1 \ldots dx^n)^\beta \rangle_{{\mathcal D}^\beta} = f f^\prime |g|^{- \beta};
\end{equation}
this is
independent of the local coordinates. 
Note that the $\beta$-th power $|g|^{\beta/2} (dx^1 \ldots dx^n)^\beta$ of the Riemannian density
is a global section of
${\mathcal D}^\beta$ with unit norm.
The metric compatible connection $\nabla^{{\mathcal D}^\beta}$ on ${\mathcal D}^\beta$ is
\begin{align} \label{6.2}
\nabla^{{\mathcal D}^\beta}_\mu \left( f (dx^1 \ldots dx^n)^\beta \right) = & \left( \partial_\mu f - \frac12 \beta (g^{\delta \epsilon} \partial_\mu g_{\delta \epsilon}) f \right) (dx^1 \ldots dx^n)^\beta
= \\
& |g|^{\beta/2} \partial_\mu \left( |g|^{-\beta/2} f \right) (dx^1 \ldots dx^n)^\beta. \notag
\end{align}

If $M$ is spin then we can consider the Dirac operator $\widehat{D}$ acting on sections of $SM \otimes {\mathcal D}^\beta$. Locally, a section can be
written as $\psi \otimes (dx^1 \ldots dx^n)^\beta$. The operator $\widehat{D}$ is defined using the connections
$\nabla^{SM}$ and $\nabla^{{\mathcal D}^\beta}$. It takes the local form
\begin{equation} \label{6.3}
\widehat{D} (\psi \otimes (dx^1 \ldots dx^n)^\beta) = |g|^{\beta/2} {D} (|g|^{-\beta/2}\psi) \otimes (dx^1 \ldots dx^n)^\beta.
\end{equation}

\section{Dirac operators on surfaces} \label{sect7}

In this section we construct Dirac operators on surfaces with irregular metrics. We prove Theorem \ref{thm3} of the introduction. When the surface is smooth except for a finite number of conical singularities, whose links are
circles of length at most $2 \pi$, we show that the
conjugated Dirac operator that we construct is equivalent to the Dirac operator constructed by Chou
\cite{Chou (1985)}. Finally, we show the equivalence of
our conjugated Dirac operator with the Dirac operator acting on spinors twisted with the quarter density bundle.

Let $M$ be a smooth compact spin surface. Let $g_0$ be a smooth
Riemannian metric on $M$. We consider Riemannian metrics $g = e^{2 \sigma} g_0$, where $e^{2 \sigma}
\in L^1(M, \dvol_{g_0})$. Letting $D_0$ denote the Dirac operator corresponding to $g_0$, if
$\sigma$ is smooth then
$D = e^{-3\sigma/2} D_0 e^{\sigma/2}$. As mentioned in Example \ref{ex5}, if $\sigma$ fails to be smooth 
(even in the case of a
single isolated cone point) then $D$ will generally not map smooth spinors to $L^2$-spinors.
Hence it is not always possible to use the smooth spinors on $M$ as a dense domain for $D$. 
On the other hand, $e^{- \sigma/2} C^\infty(SM)$ is a natural domain for $D$.

\begin{example} \label{ex6}
For a conical metric $g = |x|^{-2c} |dx|^2$, with $c < 1$, we have $e^{- \sigma/2} = |x|^{c/2}$.
Hence if $c \in (0,1)$ then the elements of $e^{- \sigma/2} C^\infty(SM)$ are spinor fields
on $\R^2$ that vanish at the origin.
\end{example}

This motivates conjugating $D$ by $e^{\sigma/2}$ and instead considering the operator
\begin{equation} \label{7.1}
\widehat{D} =  e^{\sigma/2} D e^{- \sigma/2} = e^{- \sigma} D_0.
\end{equation}
(Note that the conjugating factors $e^{\sigma/2}$ and $e^{-\sigma/2}$ may not be bounded.) It is symmetric on
smooth spinors, with
respect to the inner product $\langle \eta, \psi \rangle_{\mathcal H}  = 
\int_M \langle \eta, \psi \rangle e^\sigma \dvol_{g_0}$, since
$\langle \eta, \widehat{D} \psi \rangle_{\mathcal H} = \langle \eta, D_0 \psi \rangle_{g_0}$.

\begin{lemma} \label{lem1}
$C^\infty(SM)$ is dense in the Hilbert space
${\mathcal H} = L^2(SM, e^{\sigma} \dvol_{g_0})$.
\end{lemma}
\begin{proof}
As $e^{2 \sigma}
\in L^1(M, \dvol_{g_0})$, it follows that 
$e^{\sigma}
\in L^1(M, \dvol_{g_0})$, so
$e^{\sigma} \dvol_{g_0}$ is a Radon measure on the Borel $\sigma$-algebra of $M$. The
proof that continuous functions are dense in $L^2(M, e^{\sigma} \dvol_{g_0})$, using Urysohn's lemma,
extends by the smooth Urysohn lemma to show that 
$C^\infty(SM)$ is dense in 
${\mathcal H}$.
\end{proof}

\begin{assumption} \label{ass1}
$e^{- \sigma}
\in L^1(M, \dvol_{g_0})$
\end{assumption}

\begin{example} \label{ex7}
If $M$ has Alexandrov curvature bounded from below by a constant then $\sigma$ is bounded
from below and Assumption \ref{ass1} is satisfied.
\end{example}

\begin{example}
With reference to Example \ref{ex6}, if $c > -2$ then $e^{- \sigma}$ is locally integrable
with respect to $\dvol_{g_0}$.
\end{example}

\begin{lemma} \label{lem2}
$\widehat{D}$ maps $C^\infty(SM)$ to 
$L^2(SM, e^{\sigma} \dvol_{g_0})$.
\end{lemma}
\begin{proof}
As the square ${\mathcal H}$-norm of $\widehat{D} \psi$ is
$\int_M | \widehat{D} \psi |^2 e^\sigma \dvol_{g_0} = \int_M | D_0 \psi |^2 e^{-\sigma} \dvol_{g_0}$, the lemma follows.
\end{proof}

Let $\widehat{\del}_\pm : C^\infty(S_\pm M) \rightarrow L^2(S_\mp M, e^{\sigma} \dvol_{g_0})$ denote the restriction of $\widehat{D}$ to plus or minus spinors.

\begin{lemma} \label{lem3}
$\widehat{\del}_\pm$ is closable.
\end{lemma}
\begin{proof}
Suppose that $\{ \psi_i \}_{i=1}^\infty$ and $\{ \psi^\prime_i \}_{i=1}^\infty$ are sequences in $C^\infty(S_\pm M)$ so that
$\lim_{i \rightarrow \infty} \psi_i = \lim_{i \rightarrow \infty} \psi^\prime_i 
\stackrel{\mathcal H}{=} \psi$, 
$\lim_{i \rightarrow \infty} \widehat{\del}_ \pm \psi_i \stackrel{\mathcal H}{=} \eta$ and
$\lim_{i \rightarrow \infty} \widehat{\del}_ \pm \psi^\prime_i \stackrel{\mathcal H}{=} \eta^\prime$,
where $\psi \in L^2(S_\pm M, e^{\sigma} \dvol_{g_0})$ and
$\eta, \eta^\prime \in L^2(S_\mp M, e^{\sigma} \dvol_{g_0})$. 
We must show that $\eta= \eta^\prime$. For any $\tau \in C^\infty(S_\mp M)$, we have
\begin{align} \label{7.2}
\langle \tau, \eta \rangle_{\mathcal H} = & \lim_{i \rightarrow \infty} \langle \tau, \widehat{\del}_\pm \psi_i \rangle_{\mathcal H} = \lim_{i \rightarrow \infty} \langle \tau, D_0 \psi_i \rangle_{L^2(S_\mp M,  \dvol_{g_0})}
= \lim_{i \rightarrow \infty} \langle D_0 \tau, \psi_i \rangle_{L^2(S_\pm M, \dvol_{g_0})}
= \\
& \lim_{i \rightarrow \infty} \langle \widehat{D} \tau, \psi_i \rangle_{\mathcal H} = \langle \widehat{D} \tau, \psi \rangle_{\mathcal H}. \notag
\end{align}
Similarly, $\langle \tau, \eta^\prime \rangle_{\mathcal H} = \langle \widehat{D} \tau, \psi \rangle_{\mathcal H}$. Hence 
$\langle \tau, \eta \rangle_{\mathcal H} = \langle \tau, \eta^\prime \rangle_{\mathcal H}$ for all
$\tau \in C^\infty(S_\mp M)$. By the density of $C^\infty(SM)$ in ${\mathcal H}$, it follows that $\eta = \eta^\prime$.
\end{proof}

Let $\widehat{\del}_{\pm, min}$ be the minimal extension of $\widehat{\del}_\pm$.
It comes from pairs $(\psi, \eta) \in {\mathcal H}_\pm \times
{\mathcal H}_\mp$ so that there is a sequence $\psi_i \in C^\infty(S_\pm M)$ such that 
$\lim_{i \rightarrow \infty} (\psi_i, \widehat{\del}_\pm \psi_i) = (\psi, \eta)$ in ${\mathcal H}_\pm \times
{\mathcal H}_\mp$; then $\psi \in \Dom(\widehat{\del}_{\pm, min})$ and $\eta =  \widehat{\del}_{\pm, min} \psi$.

It will be more convenient for us to work with
the maximal extension $\widehat{\del}_{\pm, max}$ of $\widehat{\del}_\pm$. 
To define it, we first note that if $\psi \in {\mathcal H}$ then 
\begin{equation}
\int_M |\psi| \dvol_{g_0} \le
\sqrt{\int_M |\psi|^2 e^{\sigma} \dvol_{g_0} \int_M e^{- \sigma} \dvol_{g_0}} < \infty.
\end{equation}
so
$\psi \in L^1(SM, \dvol_{g_0})$. Hence it makes sense to talk about the distributional quantity
$D_0 \psi$ and consider whether it is locally square integrable with respect to $\dvol_{g_0}$.

Let $\del_{\pm,0}$ denote the distributional restriction of $D_0$ to ${\mathcal H}_\pm$.
Then the domain
$\Dom(\widehat{\del}_{\pm, max})$ consists of the $\psi \in {\mathcal H}_\pm$ such that
$\del_{\pm,0} \psi \in L^2_{loc}(M, \dvol_{g_0})$ and 
$\int_M |\del_{\pm,0} \psi|^2 e^{- \sigma} \dvol_{g_0} < \infty$.
The last inequality is formally the same as saying $\widehat{D} \psi \in L^2(SM, \dvol_g)$.

Then there is a well defined operator $\widehat{\del}_{\pm, max} : \Dom(\widehat{\del}_{\pm, max}) 
\rightarrow {\mathcal H}_\mp$
given by $\widehat{\del}_{\pm, max} \psi = e^{- \sigma} \del_{\pm,0} \psi$.
The proof of Lemma \ref{lem3} shows that $\widehat{\del}_{\pm, max}$
is closable.

The adjoint $\widehat{\del}_{\pm, max}^*$ has domain given by the $\rho \in {\mathcal H}_\mp$ so that there is some
$\tau \in {\mathcal H}_\pm$ such that for all $\psi \in \Dom(\widehat{\del}_{\pm, max})$, we have
$\langle \rho, \widehat{\del}_{\pm, max} \psi \rangle_{\mathcal H} = \langle \tau, \psi  \rangle_{\mathcal H}$;
then $\widehat{\del}_{\pm, max}^* \rho = \tau$. 
We get two self-adjoint Dirac operators $\widehat{\del}_{\pm, max} + \widehat{\del}_{\pm, max}^*$ on ${\mathcal H}$,
that {\it a priori} need not be the same. If $\sigma$ is smooth then they are the same and equal
$e^{\sigma/2} D e^{- \sigma/2}$. In this case, $e^{\sigma/2} D e^{- \sigma/2}$ acting on ${\mathcal H}$
is isometrically isomorphic to $D$ acting on $L^2(SM, \dvol_g)$.

Let $K^\frac12$ be the square root of the canonical bundle corresponding to the spin structure 
\cite[Section 2]{Hitchin (1974)} and let $\HH^*$ denote sheaf cohomology groups.

\begin{proposition} \label{prop7}
$\widehat{\del}_{\pm, max}$ is Fredholm.  Its index is zero.  There are isomorphisms $\Ker \left( \widehat{\del}_{+, max}\right) \cong \Ker \left( \widehat{\del}_{-, max}^* \right) \cong \HH^0(M, K^\frac12)$ and
$\Ker \left( \widehat{\del}_{-, max}\right) \cong \Ker \left( \widehat{\del}_{+, max}^* \right)
\cong \HH^1(M, K^\frac12)$.
\end{proposition}
\begin{proof}
We must show that  $\Ker \left( \widehat{\del}_{\pm, max} \right)$ and
$\Ker \left( \widehat{\del}_{\pm, max}^* \right)$ are finite dimensional, and
$\widehat{\del}_{\pm, max}$ has closed image.

Suppose that $\psi \in \Ker \left( \widehat{\del}_{\pm, max} \right)$. Then
$\psi \in L^1(SM, \dvol_{g_0})$ and $\del_{\pm,0} \psi = 0$ distributionally.  By elliptic regularity,
$\psi$ is smooth and $\Ker \left( \widehat{\del}_{\pm, max} \right) = 
\Ker \left( {\del}_{\pm,0} \right)$.

Suppose that $\rho \in \Ker \left( \widehat{\del}_{\pm, max}^* \right)$. Then for all $\psi \in C^\infty(S_\pm M)$, we have
$0 = \langle \rho, \widehat{\del}_{\pm, max} \psi \rangle_{\mathcal H} = \langle \rho, {\del}_{\pm,0} \psi \rangle_{L^2(SM, \dvol_{g_0})}$.
As above, $\rho$ is a distributional solution of ${\del}_{\pm,0}^* \rho = 0$. Hence $\rho$ is smooth and $\Ker \left( \widehat{\del}_{\pm, max}^* \right) = \Ker \left( {\del}_{\pm,0}^* \right)$.

The statements about the isomorphisms with the sheaf cohomology groups come from the analogous statements for 
${\del}_{\pm,0}$. The index 
vanishes by Serre duality.

It remains to show that $\widehat{\del}_{\pm, max}$ has closed image. We claim that
$\Image(\widehat{\del}_{\pm, max}) = \Ker \left( \widehat{\del}_{\pm, max}^* \right)^\perp$.
To see this, first it is automatic that 
$\Image(\widehat{\del}_{\pm, max}) \subset \Ker \left( \widehat{\del}_{\pm, max}^* \right)^\perp$.
Now suppose that $\rho \in \Ker \left( \widehat{\del}_{\pm, max}^* \right)^\perp$.
Then for all $\psi \in \Ker \left( {\del}_{\pm,0}^* \right)$, we have
$\int_M \langle e^\sigma \rho, \psi \rangle \dvol_{g_0} = 0$.
Now 
\begin{equation}
\int_M e^{\sigma} |\rho| \dvol_{g_0} \le \sqrt{
\int_M e^{\sigma} \dvol_{g_0} \int_M e^{\sigma} |\rho|^2 \dvol_{g_0}
} < \infty,
\end{equation}
so $e^\sigma \rho \in L^1(S_\mp M, \dvol_{g_0})$. Hence $e^\sigma \rho$ lies in a Sobolev space
$H^s(S_\mp M, \dvol_{g_0})$ for some $s \in \R$. By elliptic analysis, we can write
$e^\sigma \rho = {\del}_{\pm,0} \eta$ where $\eta \in H^{s+1}(S_\pm M, \dvol_{g_0})$.
Explicitly, we can take $\eta = G (e^\sigma \rho)$ where $G$ is the Green's operator for
${\del}_{\pm,0}$. It remains to show that
$\int_M |\eta|^2 e^{\sigma} \dvol_{g_0} < \infty$, i.e. that $e^{\sigma/2} \eta \in L^2(S_\pm M, \dvol_{g_0})$.

Since $e^{\sigma/2} \in L^4(M, \dvol_{g_0})$ and $e^{\sigma/2} \rho \in L^2(S_\mp M, \dvol_{g_0})$,
the H\"older inequality gives that $e^\sigma \rho \in L^{\frac43}(S_\mp M, \dvol_{g_0})$.
The Hardy-Littlewood-Sobolev inequality implies that $\eta = G (e^\sigma \rho) \in L^4(S_\pm M, \dvol_{g_0})$.
Then the H\"older inequality gives that $e^{\sigma/2} \eta \in L^2(S_\pm M, \dvol_{g_0})$.

Thus $\rho = e^{- \sigma} {\del}_{\pm,0} \eta = \widehat{\del}_{\pm,max} \eta$ with
$\eta \in {\mathcal H}_\pm$. Since 
\begin{equation}
\int_M |{\del}_{\pm,0} \eta|^2 e^{- \sigma} \dvol_{g_0} = \int_M |\rho|^2 e^{\sigma} \dvol_{g_0} < \infty,
\end{equation}
it follows that
$\eta \in \Dom(\widehat{\del}_{\pm, max})$. Hence $\Image(\widehat{\del}_{\pm, max}) = \Ker \left( \widehat{\del}_{\pm, max}^* \right)^\perp$. Since $\Ker \left( \widehat{\del}_{\pm, max}^* \right)$
is finite dimensional, $\Image(\widehat{\del}_{\pm, max})$ is closed.
\end{proof}

\begin{example} \label{ex8}
Suppose that $(M, g)$ is smooth except for a finite number of singular points $\{x_i\}_{i=1}^N$, in neighborhoods of which
the metric is conical with link length $2 \pi (1-c_i)$. The paper \cite{Chou (1985)} considers the Dirac operator on
smooth spinor fields with compact support in $M - \{x_1, \ldots x_N\}$. A special case of the results in
\cite{Chou (1985)} says that if each $c_i$ lies in $[0,1)$ then there is a unique self-adjoint extension $D_{Chou}$ to
$L^2(SM, \dvol_g)$ \cite[Remarks 3.3(3)]{Chou (1985)}. 

\begin{proposition} \label{prop8}
In this setting, $\widehat{\del}_{+, max} + \widehat{\del}_{+, max}^* = \widehat{\del}_{-, max} + \widehat{\del}_{-, max}^*$
and both are isometrically equivalent to $D_{Chou}$.
\end{proposition}
\begin{proof}
Let ${\mathcal S}$ denote the smooth spinor fields with compact support in $M - \{x_1, \ldots x_N\}$.
Then ${\mathcal S}$ is a dense domain for $D_{Chou}$. 
Letting $e^{2\sigma}$ be a conformal factor as in the beginning of this section,
there is an isometric isomorphism $\tau : L^2(SM, \dvol_g) \rightarrow {\mathcal H}$ that sends $\psi$ to $e^{\sigma/2} \psi$.
Since $\sigma$ is smooth on $M - \{x_1, \ldots x_N\}$, we know that $\tau({\mathcal S}) \subset C^\infty(SM)$.
Then we have three self-adjoint extensions of the Dirac operator on ${\mathcal S}$, namely
$D_{Chou}$, $\tau^{-1} (\widehat{\del}_{+, max} + \widehat{\del}_{+, max}^*) \tau$ and
$\tau^{-1} (\widehat{\del}_{-, max} + \widehat{\del}_{-, max}^*) \tau$. By the uniqueness of the
self-adjoint extension, they must all be the same.
\end{proof}
  
\end{example}

Let ${\mathcal D}^\frac14$ denote the quarter density bundle of $M$. As described in Section \ref{sect6}, it inherits an inner product
and a metric-compatible connection.  Using the latter, one can define the Dirac operator on smooth sections of $SM \otimes {\mathcal D}^\frac14$.

\begin{lemma} \label{lem4}
${\mathcal H}$ is isometrically isomorphic to
$L^2(SM \otimes {\mathcal D}^\frac14, \dvol_{g})$, under an isomorphism by which
$\widehat{\del}_{\pm}$ is equivalent to the Dirac operator acting on
$C^\infty(S_\pm M \otimes {\mathcal D}^\frac14)$.
\end{lemma}
\begin{proof}
It suffices to work in an isothermal coordinate chart $U$ for $g_0$, and write
$g = e^{2 \sigma} \left( (dx^1)^2 + (dx^2)^2 \right)$ (where now $\sigma$ may be redefined). As $|g| = e^{4 \sigma}$,
equation (\ref{6.1}) implies that the
square norm of $\psi \otimes (dx^1 dx^2)^{\frac14}$ on $U$ is $\int_U  |\psi|^2  e^{- \sigma} \cdot e^{2 \sigma} dx^1 dx^2 =
\int_U  |\psi|^2 e^{\sigma} dx^1 dx^2$. Comparing (\ref{6.3}) and (\ref{7.1}) shows that the two definitions of the
Dirac operator are equivalent.
\end{proof}

\section{Dirac operators in higher dimension} \label{sect8}

In this section we consider smooth manifolds of arbitrary dimension with
possibly singular Riemannian metrics.  Based on Section \ref{sect7}, we consider the 
Dirac operator $\widehat{D}$ acting on spinors twisted with the quarter density bundle.  
Proposition \ref{prop9} characterizes when $\widehat{D}$ maps smooth sections to
$L^2$-sections, in terms of the metric. In Assumption \ref{ass2} we make a slightly
stronger assumption on the metric which ensures that an integrated
Lichnerowicz formula holds. When the manifold is even dimensional, Lemma
\ref{lem7} then says that $\widehat{\partial}_{\pm}$ is closable.  We obtain
self-adjoint operators $\widehat{\partial}_{\pm, min} + \widehat{\partial}_{\pm, min}^*$.
We prove the vanishing result in Theorem \ref{thm4} of the introduction. Finally,
if $(M, g)$ is smooth except for a finite number of conical singularities, with links that are
shrinkings of the round sphere, then we show that $\widehat{\partial}_{\pm, min} + \widehat{\partial}_{\pm, min}^*$
is isometrically equivalent to the Dirac operator constructed by Chou \cite{Chou (1985)}.

We return to the setup of Section \ref{sect4}, with $M$ a smooth compact manifold equipped with a
possibly singular Riemannian metric of finite volume.  As explained in Example \ref{ex5}, in two dimensions the Dirac operator
may not send smooth spinors to $L^2$-spinors.  Lemma \ref{lem2} shows that we do not have this problem, at least on
Alexandrov surfaces,
if we consider the Dirac operator acting instead on smooth sections of
$SM \otimes {\mathcal D}^\frac14$. We now look at the $n$-dimensional version of this statement.

We will deal with the inner product on sections of $SM \otimes {\mathcal D}^\frac14$.
If $\eta, \psi \in C^\infty(SM \otimes {\mathcal D}^\frac14)$ then their inner product can be written as
$\int_M \langle \eta, \psi \rangle \sqrt{\dvol_g}$.  Note that the integrand is a
section of ${\mathcal D}$. For notation, we will write the Hilbert space completion 
of $C^\infty(SM \otimes {\mathcal D}^\frac14)$ as
$L^2(SM \otimes {\mathcal D}^\frac14, \sqrt{\dvol_g})$.

\begin{proposition} \label{prop9}
The Dirac operator $\widehat{D}$ sends $C^\infty(SM \otimes {\mathcal D}^\frac14)$ to
$L^2(SM \otimes {\mathcal D}^\frac14, \sqrt{\dvol_g})$ if and only if in coordinate charts and with respect to a
local orthonormal frame: $e_a^{\: \: \mu}$, $\Alt_{abc} \omega_{abc}$
and $4 e_a^{\: \: \mu} \omega^a_{\: \: b \mu} - e_b^{\: \: \mu} \partial_{\mu} \log |g|$ are all $L^2$
with respect to $|g|^{\frac14} dx^1 \ldots dx_n$.
\end{proposition}
\begin{proof}
From (\ref{6.3}), the operator $\widehat{D}$ is locally equivalent to $|g|^{\frac18} \circ D \circ |g|^{- \frac18}$, acting on 
$L^2(SM, |g|^{\frac14} dx^1 \ldots dx_n)$. Then acting on a smooth spinor field $\psi$,
\begin{align} \label{8.1} 
|g|^{\frac18} D (|g|^{- \frac18} \psi) = & - \sqrt{-1} \gamma^i e_i^{\: \: \mu} |g|^{\frac18} 
\left( \partial_\mu + \frac18 \omega_{jk\mu} [\gamma^j, \gamma^k] \right) (|g|^{- \frac18} \psi) \\
= & - \sqrt{-1} \gamma^i e_i^{\: \: \mu} \partial_\mu \psi -
\sqrt{-1} e_i^{\: \: \mu} 
\left(\frac18  \omega_{jk\mu} \gamma^i [\gamma^j, \gamma^k] - \frac18 \gamma^i \partial_\mu \log |g| \right)\psi. \notag
\end{align}
Since 
\begin{equation} \label{8.2}
\gamma^i [\gamma^j, \gamma^k] = 2 \left( \gamma^{[ijk]} + \delta^{ij} \gamma^k - \delta^{ik} \gamma^j \right),
\end{equation}
the proposition follows.
\end{proof}

\begin{remark} \label{rem6}
In two dimensions, the conditions of Proposition \ref{prop9} are satisfied for Alexandrov surfaces. 
In isothermal coordinates, one finds that
$4 e_a^{\: \: \mu} \omega^a_{\: \: b \mu} - e_b^{\: \: \mu} \partial_{\mu} \log |g|$ vanishes identically.
\end{remark}

\begin{lemma} \label{lem5}
$C^\infty(SM \otimes {\mathcal D}^\frac14)$ is dense in the Hilbert space
${\mathcal H} = L^2(SM \otimes {\mathcal D}^\frac14, \sqrt{\dvol_g})$.
\end{lemma}
\begin{proof}
Over a local coordinate patch $U$, the $L^2$-space is unitarily equivalent to
$L^2(SU, |g|^{\frac14} dx^1 \ldots dx^n)$. 
As $|g|^\frac12
\in L^1(U, dx^1 \ldots dx^n)$, it follows that 
$|g|^\frac14
\in L^1(U, dx^1 \ldots dx^n)$, so
$|g|^{\frac14} dx^1 \ldots dx^n$ is a Radon measure on the Borel $\sigma$-algebra of $U$. The
proof that compactly supported continuous functions are dense in $L^2(U, |g|^{\frac14} dx^1 \ldots dx^n)$, using Urysohn's lemma,
extends by the smooth Urysohn lemma to show that 
$C^\infty_c(SU)$ is dense in $L^2(SU, |g|^{\frac14} dx^1 \ldots dx^n)$. Using a partition of unity,
the lemma follows for all of $M$.
\end{proof}

Suppose that the conditions of Proposition \ref{prop9} are satisfied.

\begin{lemma} \label{lem6}
If $\eta, \psi \in C^\infty(SM \otimes {\mathcal D}^\frac14)$ then
$\int_M \langle \widehat{D} \eta, \psi \rangle \sqrt{\dvol_g} = 
\int_M \langle \eta, \widehat{D} \psi \rangle \sqrt{\dvol_g}$.
\end{lemma}
\begin{proof}
We first work in a coordinate patch $U$, with a local orthonormal frame. 
If $\eta$ and $\psi$ have compact support in $U$ then
\begin{align} \label{8.3}
& \int_U \langle \widehat{D} \eta, \psi \rangle \sqrt{\dvol_g} - 
\int_U \langle \eta, \widehat{D} \psi \rangle \sqrt{\dvol_g} = \\
& \int_U \left( \langle |g|^{\frac18} {D} (|g|^{-\frac18}\eta), \psi \rangle  - 
 \langle \eta, |g|^{\frac18} {D} (|g|^{-\frac18}\psi) \rangle \right)
|g|^{\frac14} dx^1 \ldots dx^n = \notag  \\
& \int_U \left( \langle  {D} (|g|^{-\frac18}\eta), |g|^{ -\frac18}\psi \rangle  - 
 \langle |g|^{-\frac18} \eta,  {D} (|g|^{-\frac18}\psi) \rangle \right)
|g|^{\frac12} dx^1 \ldots dx^n = \notag  \\
& \sqrt{-1} \int_U \partial_\mu \left( |g|^{\frac12} e_j^{\: \: \mu} \langle |g|^{-\frac18}\eta, \gamma^j |g|^{ -\frac18}\psi \rangle
\right)
dx^1 \ldots dx^n = \notag \\
& \sqrt{-1} \int_U \partial_\mu \left( |g|^{\frac14} e_j^{\: \: \mu} \langle \eta, \gamma^j \psi \rangle
\right)
dx^1 \ldots dx^n. \notag
\end{align}
As $e_j^{\: \: \mu}$ is $L^2$ with respect to $|g|^{\frac14} dx^1 \ldots dx^n$, it follows that
$|g|^{\frac14} e_j^{\: \: \mu}$ is $L^1$ with respect to $dx^1 \ldots dx^n$.
Hence the last term in (\ref{8.3}) vanishes.
Using a partition of unity, the proposition follows.
\end{proof}

If $g$ is smooth, consider the half density $R \sqrt{\dvol}$.  It is given by
a formula as in (\ref{3.1}), with $\sqrt{|g|}$ replaced by $|g|^\frac14$.

\begin{assumption} \label{ass2}
In local coordinates and with respect to a local orthonormal frame, $e_a^{\: \: \mu}$, 
$\omega_{abc}$ and $e_a^{\: \: \mu} \partial_\mu \log |g|$ are $L^2$ with respect to $|g|^\frac14 dx^1 \ldots dx^n$.
\end{assumption}

Note the appearance of $|g|^\frac14$ in Assumption \ref{ass2}, rather than $|g|^\frac12$.
Assumption \ref{ass2} implies that the conditions of Proposition \ref{prop9} hold.
Under Assumption \ref{ass2}, we can write $R \sqrt{\dvol}$ as in Proposition \ref{prop1}, to see that it gives a
well defined half density $\widehat{dR}$. Let $\widehat{\nabla}$ denote the covariant derivative on
sections of $SM \otimes {\mathcal D}^\frac14$. In local coordinates and with respect to a local
orthonormal frame, if $\psi \in C^\infty(SM \otimes {\mathcal D}^\frac14)$ then
\begin{equation} \label{8.4}
\widehat{\nabla}_{e_i} \psi = 
|g|^{\frac18} \nabla_{e_i} ( |g|^{-\frac18} \psi ) =
e_i^{\: \: \mu} \left( \partial_\mu \psi + \frac18 \omega_{ab\mu} [\gamma^a, \gamma^b] \psi -
\frac18 (\partial_\mu \log |g|) \psi \right).
\end{equation}
Under Assumption \ref{ass2}, the right-hand side is in $L^2$ with respect to $|g|^\frac14 dx^1 \ldots dx^n$.
Hence if $\eta, \psi \in C^\infty(SM \otimes {\mathcal D}^\frac14)$ then
$\int_M \langle \widehat{\nabla} \eta, \widehat{\nabla} \psi \rangle \sqrt{\dvol_g}$ and
$\int_M \langle \widehat{D} \eta, \widehat{D} \psi \rangle \sqrt{\dvol_g}$
are both well defined.
As in Proposition \ref{prop4}, we have
\begin{equation} \label{8.5}
\int_M \langle \widehat{D} \eta, \widehat{D} \psi \rangle \sqrt{\dvol_g} -
\int_M \langle \widehat{\nabla} \eta, \widehat{\nabla} \psi \rangle \sqrt{\dvol_g} =
\int_M \langle \eta, \psi \rangle \widehat{dR}.
\end{equation}

Suppose now that $n$ is even.
Let $\widehat{\partial}_\pm : C^\infty(S_\pm M \otimes {\mathcal D}^\frac14) \rightarrow 
L^2(S_\pm M \otimes {\mathcal D}^\frac14, \sqrt{\dvol_g})$ denote the restriction of $\widehat{D}$ to
plus or minus spinors.

\begin{lemma} \label{lem7}
$\widehat{\partial}_\pm$ is closable.
    \end{lemma}
    \begin{proof}
    The proof is similar to that of Lemma \ref{lem3}.
Suppose that $\{ \psi_i \}_{i=1}^\infty$ and $\{ \psi^\prime_i \}_{i=1}^\infty$ are sequences in 
$C^\infty(S_\pm M \otimes {\mathcal D}^\frac14)$ so that
$\lim_{i \rightarrow \infty} \psi_i = \lim_{i \rightarrow \infty} \psi^\prime_i 
\stackrel{\mathcal H}{=} \psi$, 
$\lim_{i \rightarrow \infty} \widehat{\del}_ \pm \psi_i \stackrel{\mathcal H}{=} \eta$ and
$\lim_{i \rightarrow \infty} \widehat{\del}_ \pm \psi^\prime_i \stackrel{\mathcal H}{=} \eta^\prime$,
where $\psi \in L^2(S_\pm M \otimes {\mathcal D}^\frac14, \sqrt{\dvol_g})$ and
$\eta, \eta^\prime \in L^2(S_\mp M \otimes {\mathcal D}^\frac14, \sqrt{\dvol_g})$. 
We must show that $\eta= \eta^\prime$. For any $\tau \in C^\infty(S_\mp M \otimes {\mathcal D}^\frac14)$, we have
\begin{equation} \label{8.6}
\langle \tau, \eta \rangle_{\mathcal H} = \lim_{i \rightarrow \infty} \langle \tau, \widehat{\del}_\pm \psi_i \rangle_{\mathcal H} = \lim_{i \rightarrow \infty} \langle \tau, \widehat{D} \psi_i \rangle_{\mathcal H}
= \lim_{i \rightarrow \infty} \langle \widehat{D} \tau, \psi_i \rangle_{\mathcal H}
= 
\langle \widehat{D} \tau, \psi \rangle_{\mathcal H}. 
\end{equation}
Similarly, $\langle \tau, \eta^\prime \rangle_{\mathcal H} = \langle \widehat{D} \tau, \psi \rangle_{\mathcal H}$. Hence 
$\langle \tau, \eta \rangle_{\mathcal H} = \langle \tau, \eta^\prime \rangle_{\mathcal H}$ for all
$\tau \in C^\infty(S_\mp M \otimes {\mathcal D}^\frac14)$. By the density of $C^\infty(SM \otimes {\mathcal D}^\frac14)$ in ${\mathcal H}$, it follows that $\eta = \eta^\prime$.
\end{proof}

Let $\widehat{\del}_{\pm, min}$ be the minimal extension of $\widehat{\del}_\pm$.
It comes from pairs $(\psi, \eta) \in {\mathcal H}_\pm \times
{\mathcal H}_\mp$ so that there is a sequence $\psi_i \in C^\infty(S_\pm M \otimes {\mathcal D}^\frac14)$ such that 
$\lim_{i \rightarrow \infty} (\psi_i, \widehat{\del}_\pm \psi_i) = (\psi, \eta)$ in ${\mathcal H}_\pm \times
{\mathcal H}_\mp$; then $\psi \in \Dom(\widehat{\del}_{\pm, min})$ and $\eta =  \widehat{\del}_{\pm, min} \psi$.

The adjoint $\widehat{\del}_{\pm, min}^*$ has domain given by the $\rho \in {\mathcal H}_\mp$ so that there is some
$\tau \in {\mathcal H}_\pm$ such that for all $\psi \in C^\infty(S_\pm M \otimes {\mathcal D}^\frac14)$, we have
$\langle \rho, \widehat{\del}_{\pm, min} \psi \rangle_{\mathcal H} = \langle \tau, \psi  \rangle_{\mathcal H}$;
then $\widehat{\del}_{\pm, min}^* \rho = \tau$. 
We get two self-adjoint Dirac operators $\widehat{\del}_{\pm, min} + \widehat{\del}_{\pm, min}^*$ on ${\mathcal H}$,
that {\it a priori} need not be the same.

We could also consider the maximal extension $\widehat{\del}_{\pm, max}$ of $\widehat{\del}_\pm$, as in 
Section \ref{sect7}, but in the rest of this section we will use the minimal extension.

\begin{proposition} \label{prop10}
If $\widehat{dR}$ is positive, in the sense that there is some $C>0$ such that
$\int_M f \: \widehat{dR} \ge C \int_M f \: \sqrt{\dvol_g}$ for all nonnegative continuous half densities $f$, 
then $\Ker(\widehat{\del}_{\pm, min}) = 0$.
\end{proposition}
\begin{proof}
Suppose that $\psi \in \Ker(\widehat{\del}_{\pm, min})$. By assumption, there is a sequence
$\{ \psi_i \}_{i=1}^\infty$ in $C^\infty(S_\pm M \otimes {\mathcal D}^\frac14)$
such that $\lim_{i \rightarrow \infty} \psi_i = \psi$ in ${\mathcal H}_\pm$,
and $\lim_{i \rightarrow \infty} \widehat{\del}_{\pm, min} \psi_i = 0$
in ${\mathcal H}_\mp$. Equation (\ref{8.5}) implies that
$\| \widehat{\del}_{\pm, min} \psi_i\|^2_{{\mathcal H}_\mp} \ge C \| \psi_i\|^2_{{\mathcal H}_\pm}$.
Passing to the limit gives $0 \ge C \| \psi\|^2_{{\mathcal H}_\pm}$, so $\psi = 0$. 
\end{proof}

There is a simplified version of Assumption \ref{ass2}, as follows.

\begin{assumption} \label{ass3}
In local coordinates and with respect to a local orthonormal frame, $e_a^{\: \: \mu}$ and 
$\omega_{abc}$ are $L^2$ with respect to $\dvol_g = |g|^\frac12 dx^1 \ldots dx^n$.
\end{assumption}

The previous results all go through, with obvious adjustments, if Assumption \ref{ass3} is satisfied. 
We no longer need to twist with with the quarter density bundle, so we write the Dirac operator as $D$.
As mentioned in Example \ref{ex3}, metrics with pointwise conical singularities 
satisfy Assumption \ref{ass3} if and only if $n > 2$.

\begin{example} \label{ex9}
Suppose that $(M^n, g)$, $n > 2$, is smooth except for a finite number of singular points $\{x_i\}_{i=1}^N$, in neighborhoods of which
the metric is conical as in Example \ref{ex3}. 
Then Assumption \ref{ass3} holds.
The paper \cite{Chou (1985)} considers the Dirac operator on
the smooth spinor fields with compact support in $M - \{x_1, \ldots x_N\}$.
From
\cite[Theorem (3.2)]{Chou (1985)}, if each $c_i$ lies in $[0,1)$ then there is a unique self-adjoint extension $D_{Chou}$ to
$L^2(SM, \dvol_g)$. 

\begin{proposition} \label{prop11}
In this setting, ${\del}_{+, min} + {\del}_{+, min}^* = {\del}_{-, min} + {\del}_{-, min}^*$
and both are unitarily equivalent to $D_{Chou}$.
\end{proposition}
\begin{proof}
The proof is similar to that of Proposition \ref{prop8}, but simpler because there is no twisting by the
quarter density bundle.
Let ${\mathcal S}$ denote the smooth spinor fields with compact support in $M - \{x_1, \ldots x_N\}$.
Then ${\mathcal S}$ is a dense domain for $D_{Chou}$. 
We have three self-adjoint extensions of the Dirac operator on ${\mathcal S}$, namely
$D_{Chou}$, ${\del}_{+, min} + {\del}_{+, min}^*$ and
${\del}_{-, min} + {\del}_{-, min}^*$. By the uniqueness of the
self-adjoint extension, they must all be the same.
\end{proof}

In particular, in this case we have ${\del}_{+, min}^* = {\del}_{-, min}$ and
${\del}_{-, min}^* = {\del}_{+, min}$, so the minimal extension
${\del}_{+, min} + {\del}_{-, min}$ of the
symmetric Dirac operator on $C^\infty(SM)$ is self-adjoint. Also, from
\cite{Chou (1985)}, the operator $D_{Chou}$ is Fredholm. Its index is
$\int_M \widehat{A}(TM)$, as the $\eta$-invariant of a round sphere vanishes.
\end{example}

\begin{example} \label{ex10}
Some relevant examples with conical singularities were constructed in \cite{Cecchini-Frenck-Zeidler (2024)},
with $M$ being a product of K3 surfaces.
In those examples, there is a finite number of conical singularities with links $L_i$
that are smooth spheres, but not round spheres.  The metric is biLipschitz equivalent
to a smooth metric on $M$ and the metric coefficients are $W^{1,p}$-regular for any
$p < n$. In particular, Assumption \ref{ass3} is satisfied.
The metric
has uniformly 
positive scalar curvature away from the singular points. The index of the
Dirac operator is given by
$\int_M \widehat{A}(TM) - \frac12 \sum_i \eta(L_i)$ \cite[Theorem 5.23]{Chou (1985)},
which vanishes, as in Proposition \ref{prop10}. However, the index does not
equal the $\widehat{A}$-number of $M$, which is nonzero.

Although the smooth metric on $M$ and the metric with conical singularities live on the
same manifold, they are effectively unrelated metrics,
as is witnessed by the fact that the indices of their Dirac operators are different.
On the other hand, if a metric on $M$ is biLipschitz to a smooth metric, has 
metric coefficients that are $W^{1,n}$-regular in local coordinates, and has positive scalar curvature measure,
then it
can be smoothed to a metric with positive scalar curvature \cite{Chu-Lee (2025)}.
\end{example}

\begin{remark} \label{rem7}
For metrics $g$ satisfying Assumption \ref{ass2} (or Assumption \ref{ass3}),
the existence and uniqueness of self-adjoint Dirac operators on $(M^n, g)$ extending
the Dirac operator on $C^\infty(SM \otimes {\mathcal D}^\frac14)$ (or $C^\infty(SM)$) 
is not so clear.  If
$n$ is even then Lemma \ref{lem7} implies that there is {\em some} self-adjoint Dirac operator.
It is not so clear when it is unique. The nonuniqueness result in \cite[Theorem (3.2)]{Chou (1985)}
makes it unlikely that this is always the case.
If $n$ is odd then it is not so clear when there is any self-adjoint extension of the
symmetric Dirac operator on $C^\infty(SM \otimes {\mathcal D}^\frac14)$ (or $C^\infty(SM)$). 

If $\dim(M) \ge 3$, $g$ is biLipschitz to a smooth metric, and the metric coefficients are $W^{1,n}$-regular in
local coordinates, then the Dirac operator on smooth spinors is essentially self-adjoint
\cite[Section 3]{Bartnik-Chrusciel (2005)}. On the other hand, a metric with isolated conical
singularities as in Example \ref{ex3} has local coordinates in which $g$ is biLipschitz to a smooth metric, and
the metric coefficient are $W^{1,p}$-regular for $p < n$, but are not $W^{1,n}$-regular.  If $c_i$
is very negative then the link of the cone is a large sphere. In such a case, one may expect issues with
essential self-adjointness \cite{Chou (1985)}. 
\end{remark}

\section{Harmonic coordinates} \label{sect9}

In this section we work with harmonic coordinates.  We show that the formula for the scalar curvature
measure simplifies and we prove Theorem \ref{thm5} of the introduction.

Let $M$ be a smooth $n$-dimensional manifold and let $g$ be a smooth metric on $M$. 
We can find local harmonic coordinates $\{x^\mu\}$ and in these coordinates, the scalar curvature measure is
given by
\begin{equation} \label{9.1}
R \dvol_g = \left( - \frac12 \triangle_g \log |g| + g^{\mu \nu} \Gamma^\alpha_{\: \: \beta \mu} 
\Gamma^\beta_{\: \: \alpha \nu} \right) \dvol_g.
\end{equation}

It can be expressed in a more geometric form
\cite{Chern-Goldberg (1975)}
by considering the coordinates as giving a
map $x : U \rightarrow \R^n$ with components $\{x^\mu\}_{\mu=1}^n$.
Let $\{y^i\}_{i=1}^n$ be coordinates on $U$, with metric tensor $h_{ij}$.
The tension field, the covariant derivative
of $dx$, has components $\{ T^\alpha_{\: \: ij} \}$.
The harmonicity condition is $h^{ij} T^\alpha_{\: \: ij} = 0$. 
Put $u = \frac{x^* \dvol_{\R^n}}{\dvol_h}$. Writing 
$dx(\partial_i) = e_i^{\: \: \mu} \partial_\mu$ and
$(dx)^{-1}(\partial_\mu) = e_\mu^{\: \: i} \partial_i$, the formula for the scalar measure is
\begin{equation} \label{9.2}
R \dvol_h = \left(  \triangle_h \log u + 
e_\alpha^{\: \: i} e_\beta^{\: \: j} h^{kl}
T^\alpha_{\: \:jk} T^\beta_{\: \: il} \right) \dvol_h.
\end{equation}
As $u = \frac{\det(e_i^{\: \: \mu})}{\sqrt{|h|}}$, we can also write
\begin{equation} \label{9.3}
R \dvol_h = \left(  \triangle_h \left( \log \det(e_i^{\: \: \mu}) - \frac12 \log |h| \right)+ 
e_\alpha^{\: \: i} e_\beta^{\: \: j} h^{kl}
T^\alpha_{\: \:jk} T^\beta_{\: \: il} \right) \dvol_h.
\end{equation}
If we take $\{ x^{\mu} \}$ to be coordinates then the metric $g$ in these
coordinates is given by $g^{\mu \nu} = \langle \nabla x^\mu, \nabla x^\nu \rangle = 
h^{ij} e_i^{\: \: \mu} e_j^{\: \: \nu}$ and (\ref{9.3}) reduces to (\ref{9.1}).

In view of (\ref{9.1}), we make the following assumption.
\begin{assumption} \label{ass4}
Outside of a closed set ${\mathcal S}$ of measure zero, $M$ is covered by coordinate charts $U$ in which 
$g_{\mu \nu}$ and $g^{\mu \nu}$ are measurable, 
$g_{\mu \nu}$ is locally $L^1$ and has measurable first
distributional derivatives, and $g^{\mu \nu} \Gamma^\alpha_{\: \: \mu \nu} = 0$.
\end{assumption}

\begin{proposition} \label{prop12}
If $\sqrt{|g|} g^{\mu \nu} \log |g|$ and 
$\sqrt{|g|} g^{\mu \nu} \Gamma^\alpha_{\: \: \beta \mu} 
\Gamma^\beta_{\: \: \alpha \nu}$ are $L^1$ with respect to $dx^1 \ldots dx^n$ then the
formula (\ref{9.1}) gives a 
well defined distribution $dR$ on $U$.
\end{proposition}
\begin{proof}
In harmonic coordinates, $\triangle_g = g^{\mu \nu} \partial_\mu \partial_\nu$.
Formally, if $f \in C^\infty_c(U)$ then
\begin{align} \label{9.4}
\int_U f \triangle_g \log |g| \dvol_g = & \int_U (\triangle_g f)  \log |g| \dvol_g =
\int_U (g^{\mu \nu} \partial_\mu \partial_\nu f)  \log |g| \dvol_g \\
& =
\int_U (\partial_\mu \partial_\nu f)  \sqrt{|g|} g^{\mu \nu} \log |g| \: dx^1 \ldots dx^n. \notag
\end{align}
We can use the last term in this equation to define $(\triangle_g \log |g|) \dvol_g$ as a distribution.
Then $dR$ makes sense as a distribution.
\end{proof}

\begin{proposition} \label{prop13}
Under the assumptions of Proposition \ref{prop12}, suppose that the distribution $dR$ is a measure.
Let $U$ be a harmonic coordinate patch and let $Z \subset U$ be the image of a smooth embedding of
the closed $k$-ball in $U$ for $k < n$. Let $N_\epsilon(Z)$ be the $\epsilon$-distance neighborhood of $Z$ with respect to
the Euclidean coordinates of $U$. If $\int_{N_\epsilon(Z)} \parallel g^{\mu \nu}  \parallel \cdot
|\log |g|| \dvol_g = o(\epsilon^2)$ as $\epsilon \rightarrow 0$ then
$dR$ vanishes on $\Int(Z)$. That is, for any Borel subset $W$ of $\Int(Z)$, we have $dR(W) = 0$.
\end{proposition}
\begin{proof}
Since $\sqrt{|g|} g^{\mu \nu} \Gamma^\alpha_{\: \: \beta \mu} 
\Gamma^\beta_{\: \: \alpha \nu}$ is $L^1$ with respect to $dx^1 \ldots dx^n$, it is enough to concentrate on the
$- \frac12 (\triangle_g \log |g|) \dvol_g$ term in (\ref{9.1}). By our assumptions, it is a regular measure on $U$.
Choose $z \in \Int(Z)$, let $B$ be the $\delta$-ball around $z$ in Euclidean coordinates and let $F$ be a smooth function on $U$
with support in $B$. Let $d_Z^2 \in C(B)$ be the square of the Euclidean distance from $Z$; it is smooth if we take $\delta$ sufficiently
small. Let $\phi \in C^\infty_c([0, \infty))$ be a smooth nonincreasing
function with support in $[0,1]$, and $\phi(0) = 1$. If $\epsilon > 0$, put $f_\epsilon = F \phi(d_Z^2/\epsilon^2)$. 
For small $\epsilon$, we have 
\begin{align} \label{9.5}
\int_U f_\epsilon (\triangle \log |g|) \dvol_g = & \int_U (g^{\mu \nu} \partial_\mu \partial_\nu f_\epsilon) (\log |g|) 
\sqrt{|g|} dx^1 \ldots dx^n \\
= & \int_{B \cap N_\epsilon(Z)} (g^{\mu \nu} \partial_\mu \partial_\nu f_\epsilon) (\log |g|) 
\sqrt{|g|} dx^1 \ldots dx^n. \notag
\end{align}
Now $|\partial_\mu \partial_\nu f_\epsilon| \le \const \epsilon^{-2}$. Taking $\epsilon \rightarrow 0$ in (\ref{9.5})
shows that $\int_Z F dR = 0$. Given the arbitrariness in the choices of $z$, $B$ and $F$, it follows that
$dR$ vanishes on $Z$.
\end{proof}

\begin{example} \label{ex11}
Consider a $n$-dimensional cone whose link is a sphere with metric $(1-c)^2$ times the round metric $h$,
where $c \in [0,1)$. Put
$\alpha = \frac12 \left( \sqrt{
(n-2)^2 + 4 (1-c)^{-2} (n-1)
} - (n-2) \right)$.
Consider the metric given in polar coordinates on $\R^n$ by
$g = R^{\frac{2}{\alpha} - 2} \left( dR^2 + \alpha^2 (1-c)^2 R^2 h \right)$.
When rewritten in Cartesian coordinates, this gives the cone metric in harmonic coordinates.
With respect to Proposition \ref{prop13}, let $Z$ be the origin. One can check that
$\int_{N_\epsilon(Z)} \parallel g^{\mu \nu}  \parallel \cdot
| \log |g|| \dvol_g = o(\epsilon^2)$ as $\epsilon \rightarrow 0$ 
if $n \ge 3$. On the other hand, if $n = 2$ then we know from Example \ref{ex3}, which is
already in harmonic coordinates, that $dR$ has singular support at the origin.
\end{example}

       \end{document}